\documentclass[11pt,a4paper,leqno]{amsart}

\usepackage[latin1]{inputenc}
\usepackage[T1]{fontenc}
\usepackage{amsfonts}
\usepackage{amsmath}
\usepackage{amssymb}
\usepackage{eurosym}
\usepackage{mathrsfs}
\usepackage{palatino}
\usepackage{color}
\usepackage{xcolor}
\usepackage{esint}
\usepackage{url}
\usepackage{hyperref}

\newcommand{\R}{\mathbb{R}}

\newcommand{\C}{\mathbb{C}}

\newcommand{\N}{\mathbb{N}}
\newcommand{\No}{\mathbb{N}_0}
\newcommand{\Z}{\mathbb{Z}}

\numberwithin{equation}{section}

\newcommand{\ud}[0]{\,\mathrm{d}}

\newcommand{\esssup}[0]{\operatornamewithlimits{ess\,sup}}

\newcommand{\dist}[0]{\operatorname{dist}}

\newcommand{\abs}[1]{|#1|}
\newcommand{\babs}[1]{\big|#1\big|}
\newcommand{\Babs}[1]{\Big|#1\Big|}
\newcommand{\norm}[2]{\|#1\|_{#2}}

\newcommand{\Norm}[2]{\|#1\|_{#2}}

\newcommand{\ave}[1]{\langle #1\rangle}
\newcommand{\bave}[1]{\big\langle #1\big\rangle}

\newcommand{\BMO}[0]{\operatorname{BMO}}
\newcommand{\supp}[0]{\operatorname{supp}}

\newcommand{\prob}[0]{\mathbb{P}}
\newcommand{\Exp}[0]{\mathbb{E}}

\theoremstyle{plain}
\newtheorem{thm}{Theorem}[section]
\newtheorem{lem}[thm]{Lemma}

\theoremstyle{definition}
\newtheorem{defn}[thm]{Definition}

\theoremstyle{remark}
\newtheorem{rem}[thm]{Remark}

\title{Compactness of bilinear singular integral with mild kernel regularity}

\author{Jinsong Li}

\address{Center for Applied Mathematics, Tianjin University, Weijin Road 92, 300072 Tianjin, China}
\email{ljs@tju.edu.cn}

\makeatletter
\@namedef{subjclassname@2020}{%
  \textup{2020} Mathematics Subject Classification}
\makeatother


\begin{document}

\allowdisplaybreaks

\begin{abstract}
We establish a compact bilinear $T1$ theorem for Calder\'on-Zygmund operators with mild kernel regularity.
 The exponent obtained here matches the best-known threshold in the sufficient condition for the classical bilinear $T1$ theorem. We also introduce an equivalent unified testing condition WCP$^{*}$.
\end{abstract}

\maketitle
\section{Introduction and main theorem}
Over the past several decades, the singular integral theory has been significantly developed, becoming an indispensable tool in many other fields, including real and complex analysis, partial differential equations and geometric measure theory. Usually a singular integral operator can be represented as
\[
T(f)(x)=\int K(x,y)f(y)\ud y, \,\text{ whenever }x\notin \supp f,
\]
where the kernel $K$ satisfies
\[
\abs{K(x,y)-K(x',y)}+\abs{K(y,x)-K(y,x')}\lesssim\frac{\abs{x-x'}^\alpha}{\abs{x-y}^{n+\alpha}},\text{ whenever }\abs{x-x'}\leq\frac{\abs{x-y}}{2},
\]
 where $0<\alpha\leq1$. In \cite{CZ1952}, Calder\'on and Zygmund established the classical Calder\'on-Zygmund principle, reducing the general $L^p$ $(1<p<\infty)$ boundedness of this class of operators to the case $p=2$. Subsequently, David and Journ\'e \cite{DJ1984} provided a necessary and sufficient condition for the $L^2$ boundedness of such operators--the $T1$ theorem. Later, Hyt\"onen \cite{H2012} provided a modern proof of $T1$ theorem, in fact, it is an enhanced version--the so-called dyadic representation theorem. 
For more recent advances, see \cite{DJ1984,KS1999,GT2002,LMO2019,LMV2021}.

The operators above can fit into the following family as the special case $\omega(t)=t^\alpha$. To go further, we now consider a larger class of operators, namely those whose kernels satisfy the following Dini-type conditions:
\[
\abs{K(x,y)-K(x',y)}+\abs{K(y,x)-K(y,x')}\lesssim\omega\left(\frac{\abs{x-x'}}{\abs{x-y}}\right)\frac{1}{\abs{x-y}^n},
\]
whenever $\abs{x-x'}\leq\abs{x-y}/2$,
where $\omega$ is a modulus of continuity, satisfies
\[
\int_{0}^1\omega(t)\Big(1+\log\frac{1}{t}\Big)^\beta\frac{\ud t}{t}<\infty\,\text{ for some }\beta\geq0.
\]
 It was proved that the classical Calder\'on-Zygmund principle remains valid with $\beta=0$, see e.g.\cite{Stein1993}. However, the corresponding $T1$ theorem for singular integrals with such kernel is quite challenging. Currently the best known result in this direction  is $\beta=1/2$, which appears implicit in Figiel \cite{F1990} and was explicitly formulated by Deng, Yan and Yang \cite{DYY1998}. Later, this result was extended by Grau de la Herr\'an and Hyt\"onen \cite{GH2018} to the non-homogeneous scenario, and by Airta, Martikainen and Vuorinen \cite{AMV2022} to the multilinear and multiparameter settings while retaining $\beta=1/2$.

Compactness is also an important topic in harmonic analysis. The systematic study of compactness for commutators of singular integral operators was initiated by Uchiyama \cite{Uchi1978}, who proved that for a non-degenerate Calder\'on-Zygmund operator $T$, the commutator $[b,T]$ is compact on $L^p$ ($1<p<\infty$) if and only if the symbol $b\in\mathrm{CMO}$. This result has been extended to various settings, including the bilinear, weighted, and bi-parameter cases; see, for example, \cite{BT2013,BDMT2015,HLTY2023,MT2024}. Regarding the compactness of singular integral operators, Villarroya \cite{V2015} first established a characterization of compact Calder\'on-Zygmund operators, which was later extended to the endpoint setting in \cite{OV2017}. More recently, Mitkovski and Stockdale \cite{MS2026} obtained a new formulation of the compact $T1$ theorem, while Fragkos, Green, and Wick \cite{FGW2026} proved a wavelet $T1$ theorem characterizing the compactness of multilinear Calder\'on-Zygmund operators. In the bilinear setting, Cao, Liu, Si, and Yabuta \cite{cao2024} established a compact $T1$ theorem based on a dyadic representation of compact bilinear Calder\'on-Zygmund operators.

It is natural to ask whether the compact bilinear $T1$ theorem in \cite{cao2024} can be extended to kernels with a general modulus of continuity. However, a direct extension of the dyadic representation in \cite{cao2024} leads to a complexity loss of order $k$, corresponding to the Dini-type condition with $\beta=1$. In this paper, we overcome this obstacle by reorganizing the dyadic representation and introducing a new family of modified compact bilinear shifts. This reorganization reduces the complexity growth from order $k$ to order $k^{1/2}$, which is the key step that enables the compact $T1$ theorem under the $\log\frac{1}{2}$-Dini condition (i.e., the Dini-type condition with $\beta=1/2$). This exponent $1/2$ coincides with the best currently known regularity threshold in the bilinear $T1$ theorem for boundedness. In addition, we obtain an alternative characterization of compactness, which we call WCP$^*$. This condition may be regarded as a slightly weaker analogue of the "$L^2$-condition" appearing in \cite{BLL2025} and in \cite[Theorem~2, statement~B]{FGW2026}. Our main results are summarized in the following theorem.
\begin{thm}\label{main thm}
    Let $T$ be a bilinear operator associated with a compact bilinear Calder\'on-Zygmund kernel with $\log\frac{1}{2}$-Dini condition, then the following assertions are equivalent:
    \begin{enumerate}
    \item $T,T^{*1},T^{*2}$ are compact from $L^{p_1}(\R^n)\times L^{p_2}(\R^n)$ to $L^p(\R^n)$ for all $1<p_1,p_2\leq\infty$ with $1/p=1/p_1+1/p_2>0$;
    \item $T$ satisfies the WCP and 
      \[
    S(1,1)\in\mathrm{CMO}\,\text{ for all }S\in\{T,T^{*1},T^{*2}\};
    \]
    \item $T,T^{*1},T^{*2}$ satisfy the WCP$^*$.
    \end{enumerate}
\end{thm}
Here WCP means \textit{Weak Compactness Property}, which will be introduced in definition \ref{WCP} and \ref{WCP*}.
\subsection*{Notation}
Throughout this paper, we write $A\lesssim B$ if there exists an absolute constant $c$ such that $A\leq cB$, and $A\lesssim_k B$ denotes that the implicit constant may depend on the parameter $k$. Let $\N_0=\N\cup\{0\}$. For a vector $x\in\R^n$, we will use the norm $\abs{x}=\max\{\abs{x_j}:1\leq j\leq n\}.$ 
\section{Preliminary}
\subsection{Compact bilinear Calder\'on-Zygmund operators}
We begin with following definitions, based on the setup of \cite{cao2024}, which is originally from \cite{V2015,OV2017}:
\begin{defn}
    Let $\mathcal{F}$ consist of all triples $(F_1,F_2,F_3)$ of bounded functions, where $F_i:[0,\infty)\to [0,\infty)$ satisfy
    \begin{equation*}
        \lim_{t\to 0}F_1(t)= \lim_{t\to \infty}F_2(t)= \lim_{t\to \infty}F_3(t)=0.
    \end{equation*}
    Let $\mathcal{F}_0$ be the collection of all bounded functions $F:\mathcal{Q}\to[0,\infty)$ satisfying
    \begin{equation*}
        \lim_{\ell(Q)\to 0}F(Q)= \lim_{\ell(Q)\to \infty}F(Q)= \lim_{\abs{c_Q}\to \infty}F(Q)=0,
    \end{equation*}
where $\mathcal{Q}$ denotes the collection of all cubes in $\R^n$ with sides parallel to the axes and $c_Q$ denotes the center of $Q$.
\end{defn}
\begin{defn}
A function $K(x,y,z):(\R^{n})^3\backslash \{(x,y,z):x=y=z\}\to \C$ is a \textit{compact bilinear Calder\'on-Zygmund kernel with $\log\frac{1}{2}$-Dini condition} if it satisfies
\begin{equation}
    \abs{K(x,y,z)}\lesssim\frac{F(x,y,z)}{(\abs{x-y}+\abs{x-z})^{2n}},
\end{equation}
\begin{equation}
    \abs{K(x,y,z)-K(x',y,z)}\lesssim\omega\left(\frac{\abs{x-x'}}{\abs{x-y}+\abs{x-z}}\right)\frac{F(x,y,z)}{(\abs{x-y}+\abs{x-z})^{2n}},
\end{equation}
whenever $\abs{x-x'}\leq\max\{\abs{x-y},\abs{x-z}\}/2$, where 
\begin{equation*}
    F(x,y,z):=F_1(\abs{x-y}+\abs{x-z})F_2(\abs{x-y}+\abs{x-z})F_3(\abs{x+y}+\abs{x+z}),
\end{equation*}
for $(F_1,F_2,F_3)\in\mathcal{F}$ and $\omega$ is a modulus of continuity: an increasing, subadditive function $\omega$ with $\omega(0)=0$ and satisfies
\begin{equation*}
    \int_{0}^1\omega(t)\left(1+\log\frac{1}{t}\right)^{1/2}\frac{\ud t}{t}<\infty,
\end{equation*}
together with other two symmetrical regularity conditions.
\end{defn}
\begin{defn}
    Let $T$ be a bilinear operator defined initially on finite linear combinations of characteristic functions of cubes in $\R^n$, we say that $T$ is \textit{associated with a compact bilinear Calder\'on-Zygmund kernel with $\log\frac{1}{2}$-Dini condition} if there exists a compact bilinear Calder\'on-Zygmund kernel with $\log\frac{1}{2}$-Dini condition $K$ such that 
    \[
    \ave{T(f_1,f_2),f_3}=\int K(x,y,z)f_1(y)f_2(z)f_3(x)\ud y\ud z\ud x
    \]
    for $\supp(f_i)\cap\supp(f_j)=\emptyset$ with some $i\neq j$. And the adjoints $T^{*1},T^{*2}$ are defined via
    \[
    \ave{T(f_1,f_2),f_3}=\ave{T^{*1}(f_3,f_2),f_1}=\ave{T^{*2}(f_1,f_3),f_2}.
    \]
\end{defn}
\begin{defn}\label{WCP}
    We say that operator $T$ satisfies the \textit{weak compactness property} (abbreviated as WCP ) if there exists $F\in\mathcal{F}_0$,
    \begin{equation*}
        \abs{\ave{T(1_Q,1_Q),1_Q}}\lesssim F(Q)\abs{Q}\text{ for all cubes }Q\in\mathcal{Q}.
    \end{equation*}
\end{defn}
\begin{defn}\label{WCP*}
    We say that operator $T$ satisfies the \textit{weak compactness property}$^*$ (abbreviated as WCP$^*$ ) if there exists $F\in\mathcal{F}_0$,
    \begin{equation*}
        \norm{T(1_Q,1_Q)1_Q}{L^2}\lesssim F(Q)\abs{Q}^{1/2}\text{ for all cubes }Q\in\mathcal{Q}.
    \end{equation*}
\end{defn}
\subsection{CMO and Weights}
\begin{defn}
    We say that a locally integrable function $f\in\BMO(\R^n)$, if 
    \begin{equation*}
        \norm{f}{\BMO}:=\sup_{Q\in\mathcal{Q}}\frac{1}{\abs{Q}}\int_Q\abs{f(x)-\ave{f}_Q}\ud x<\infty,
    \end{equation*}
     where $\ave{f}_I$ denotes $\frac{1}{\abs{I}}\int_If$. The space $\mathrm{CMO}(\R^n)$ is defined as the closure of $C_c^{\infty}(\R^n)$ in $\BMO(\R^n)$. For our purpose, it will be more convenient to use the following characterization.
\end{defn}
\begin{lem}\cite[Lemma 3]{Uchi1978}\label{CMO}
    Let $f\in\BMO(\R^n)$, then $f\in\mathrm{CMO}(\R^n)$ if and only if $f$ satisfies 
    \begin{equation*}
        \inf_a\frac{1}{\abs{Q}}\int_Q\abs{f(x)-a}\ud x\to 0,
    \end{equation*}
    as $\ell(Q)\to 0$ or $\ell(Q)\to \infty$ or $\abs{c_Q}\to\infty$.
\end{lem}
\begin{defn}
    Given $1<p<\infty$, we define the Muckenhoupt weight $A_p$ as the collection of all weights on $\R^n$ satisfying
\[
[w]_{A_{p}}:=\sup_{Q\in\mathcal{Q}}\ave{w}_Q\ave{w^{1-p'}}_Q^{p-1}<\infty,\quad A_{\infty}:=\cup_{p>1}A_{p}.
\]
\end{defn}
\begin{defn}
    Given $\vec{p}=(p_1,\dots,p_m)$ with $1\leq p_1,\dots,p_m\leq\infty$ and $1/p=\sum_{i=1}^m1/p_i$, we say $\vec{w}=(w_1,\dots,w_m) $ is a $m$-linear $A_{\vec{p}}$ weight in $\R^n$ if 
\begin{align*}
  [\vec{w}]_{A_{\vec{p}}}=\sup_{Q\in\mathcal{Q}}\ave{w}_{p,Q}\prod_{i=1}^{m}\ave{w_i^{-1}}_{p_i',Q}<\infty,\quad w=\prod_{i=1}^{m}w_i,
\end{align*} 
where $\ave{w}_{p,Q}:=\ave{w^p}_Q^{1/p}$ for $0< p<\infty$ and $\ave{w}_{\infty,Q}:=\esssup_{x\in Q}w(x)$.
\end{defn}
The basic multilinear weighted theory can be found in e.g.\cite{LOPTTR2009}. The extrapolation of compactness was introduced by Hyt\"onen and Lappas in \cite{HL2023} and was extended to the multilinear setting in \cite{HL2022,COY2022,cao2024}. In the multilinear setting, the following extrapolation theorem for compactness allows one to reduce compactness in the full range to compactness and weighted boundedness in Banach range. 
\begin{thm}\cite[Theorem 1.9]{cao2024}\label{extrapolation of compactness}
    Assume that $T$ is a bilinear operator such that 
    \[
    \text{ T is compact from }L^{p_1}(u_1^{p_1})\times L^{p_2}(u_2^{p_2})\text{ to }L^p(u^p)
    \]
    for some $1/p=1/p_1+1/p_2>0$ with $p_1,p_2\in[1,\infty]$ and for some $(u_1,u_2)\in A_{(p_1,p_2)}$, where $u=u_1u_2$; and 
    \[
    \text{ T is bounded from }L^{q_1}(v_1^{q_1})\times L^{q_2}(v_2^{q_2})\text{ to }L^q(v^q)
    \]
    for some $1/q=1/q_1+1/q_2>0$ with $q_1,q_2\in[1,\infty]$ and for all $(v_1,v_2)\in A_{(q_1,q_2)}$, where $v=v_1v_2$. Then 
    \[
    \text{ T is compact from }L^{r_1}(w_1^{r_1})\times L^{r_2}(w_2^{r_2})\text{ to }L^r(w^r)
    \]
    for all $1/r=1/r_1+1/r_2>0$ with $r_1,r_2\in(1,\infty]$ and for all $(w_1,w_2)\in A_{(r_1,r_2)}$, where $w=w_1w_2$.
\end{thm}
\subsection{Dyadic grids}
Begin with the following standard dyadic grid $\mathcal{D}$ in $\R^n$, 
\[
\mathcal{D}:=\{2^{-k}([0,1)^n+m):k\in\Z,m\in\Z^n\}.
\] 
Given $I\in\mathcal{D}$, we use the following notations throughout this paper.
\begin{enumerate}
    \item $\ell(I)$ denotes the sidelength of $I$;
    \item $I^{(k)}$ denotes the $k$-th parent of $I$, $I\subset I^{(k)}$ and $\ell(I^{(k)})=2^k\ell(I)$;
    \item $\mathrm{ch}(I)$ is the collection of the children of $I$, $\mathrm{ch}(I)=\{J\in\mathcal{D}:J^{(1)}=I\}$;
    \item $E_If=\ave{f}_I1_I$ is the conditional expectation on $I$;
    \item $\Delta_If=\sum_{J\in\mathrm{ch}(I)}E_Jf-E_If$ is the martingale difference on $I$.
\end{enumerate}
The shifted lattice is defined by 
\begin{align*}
  \mathcal{D}(\sigma):=\Big\{L+\sigma:=L+\sum_{i:2^{-i}<\ell(L)}2^{-i}\sigma_i: L\in\mathcal{D}\Big\}\text{ for }\sigma=(\sigma_i)_{i\in\Z}\in(\{0,1\}^n)^{\Z}.
\end{align*}
$\mathcal{D}(\sigma)$ is still a dyadic grid, so it inherits the notations above. For each $\sigma\in(\{0,1\}^n)^{\Z}$, let $\mathcal{D}_N(\sigma)$ denote the $N$-truncation of $\mathcal{D}(\sigma)$, namely
    \begin{equation*}
       \mathcal{D}_N(\sigma):=\Big\{I\in\mathcal{D}(\sigma):I\cap[-2^N,2^N)^n\neq\emptyset\text{ and }2^{-N}\leq \ell(I)\leq 2^{N}\Big\}.
    \end{equation*}

Let $\prob_{\sigma}$ be the product probability measure on $(\{0,1\}^n)^{\Z}$. We say that $G$ is $k$-good or $G\in\mathcal{D}(\sigma,k)$ if $G\in\mathcal{D}(\sigma)$ and
$\mathrm{dist}\big(G,\partial G^{(k)}\big)\geq\ell\big(G^{(k)}\big)/{4}=2^{k-2}\ell(G).$
From \cite{GH2018}, it follows that
\[
\prob_{\sigma}\big(\{\sigma:L\in\mathcal{D}(\sigma,k)\}\big)=\frac{1}{2^n},\,\text{ for all }k\geq 2,
\]
and $G\in\mathcal{D}(\sigma,k)$ implies that $(G\dotplus l)^{(k)}=G^{(k)}$ for $l\in\Z^n$ with $\abs{l}\leq2^{k-2}$, where $G\dotplus l=G+l\ell(G)$.

Furthermore, for all $k\geq2$, the $k$-goodness of a shifted cube and its position are independent. This follows from the fact that the position of $L+\sigma$ depends only on the values of $\sigma_i$ with $2^{-i}<\ell(L)$, whereas its goodness depends only on the remaining $\sigma_i$. 

For $f\in L^2$, we have the following dyadic martingale difference decomposition 
\[
f=\sum_{I\in\mathcal{D}(\sigma)}\Delta_If.
\]
This decomposition can be further expressed in terms of Haar functions, which are defined as follows.
\begin{defn}
    Let $h_I^0,h_I^1$ be standard Haar functions on interval $I\in\R$, namely
    $$h_I^0=\abs{I}^{-1/2}1_I\text{ and }h_I^1=\abs{I}^{-1/2}(1_{I^{-}}-1_{I^{+}}),$$
    where $I^-,I^+$ denote the left and right halves of $I$ respectively. In higher dimension, for cube $I=I_1\times\cdots\times I_n\in\R^n$, the Haar function $\{h_I^\eta\}_{\eta\in\{0,1\}^n}$ may be defined by 
    \[
    h_I^\eta:=h_{I_1}^{\eta_1}\otimes\cdots\otimes h_{I_n}^{\eta_n},\,\eta=(\eta_1,\dots,\eta_n).
    \]
    A direct calculation gives  
    \[
    f=\sum_{I\in\mathcal{D}(\sigma)}\sum_{\eta\neq0}\ave{f,h_I^\eta}h_I^\eta.
    \]
    Usually, we suppress the $\eta$ summation and just write $f=\sum_{I\in\mathcal{D}(\sigma)}\ave{f,h_I}h_I$. 
\end{defn}
\begin{defn}
    Given $N\in\N$ and $\sigma\in(\{0,1\}^n)^{\Z}$, we define the \textit{projection operator} and its \textit{orthogonal operator} by 
    \begin{equation*}
        P_N^\sigma f:=\sum_{Q\in\mathcal{D}_N(\sigma)}\ave{f,h_Q}h_Q\,\text{ and }\,(P_N^\sigma)^\bot f:=f- P_N^\sigma f.
    \end{equation*}
\end{defn} 
As a corollary of the square functions estimate in \cite[Theorem 2.1]{Wilson2008}, we have
\begin{lem}
    The projection operator and its orthogonal operator defined above are bounded from $L^p\to L^p$ for all $1<p<\infty$.
\end{lem}
This projection operator yields a characterization for compactness, see e.g. \cite[Theorem 1.6]{cao2024}. Although this result also follows as a special case of \cite[Theorem B]{MSWW2023}, we provide a different proof here for completeness.
\begin{lem}\label{projection}
    Let $T$ be a bilinear operator, $1/p=1/p_1+1/p_2>0$ with $p,p_1,p_2\in(1,\infty]$ and $\sigma\in(\{0,1\}^n)^{\Z}$, if $T$ is bounded from $L^{p_1}\times L^{p_2}\to L^p$, then $T$ is compact from $L^{p_1}\times L^{p_2}\to L^p$ if and only if $\lim_{N\to\infty}\norm{(P_N^\sigma)^\bot T}{L^{p_1}\times L^{p_2}\to L^p}=0$.
\end{lem}
\begin{proof}
    The sufficiency part is immediate. Indeed, observe that $P_N^\sigma T$ is  finite rank, so of course it is compact. Then $\lim_{N\to\infty}\norm{(P_N^\sigma)^\bot T}{L^{p_1}\times L^{p_2}\to L^p}=0$ implies that 
    \begin{equation*}
        P_N^\sigma T\to T\text{ as }N\to\infty.
    \end{equation*}
    Therefore $T$ is compact from $L^{p_1}\times L^{p_2}\to L^p$. Conversely, suppose $T$ is compact. Then for any $\varepsilon>0$, there exists a collection of finitely many functions 
    $\{g_k\}_{k=1}^m\subset L^{p}$
     such that 
\begin{equation*}
    \Big\{T(f_1,f_2):\norm{f_j}{L^{p_j}}\leq1,j=1,2\Big\}\subset\bigcup_{k=1}^m B(g_k,\varepsilon).
\end{equation*}
Then for any fixed $f_1,f_2$ with $\norm{f_1}{L^{p_1}}\leq1$ and $\norm{f_2}{L^{p_2}}\leq1,$ there exists some $g_k$ satisfying
\[
\norm{T(f_1,f_2)-g_k}{L^p}\leq\varepsilon.
\]
Hence,
\begin{align*}
    &\norm{T(f_1,f_2)-P_N^\sigma T(f_1,f_2)}{L^p}\\
    &\leq\norm{T(f_1,f_2)-g_k}{L^p}+\norm{P_N^\sigma\big(T(f_1,f_2)-g_k\big)}{L^p}+\norm{P_N^\sigma g_k-g_k}{L^p}\lesssim\varepsilon+\norm{P_N^\sigma g_k-g_k}{L^p}.
\end{align*}
Taking the supremum on both sides, we obtain
\begin{equation*}
    \sup_{\norm{f_j}{L^{p_j}}\leq 1}\norm{T(f_1,f_2)-P_N^\sigma T(f_1,f_2)}{L^p}\lesssim\varepsilon+\sum_{k=1,\dots,m}\norm{P_N^\sigma g_k-g_k}{L^p},
\end{equation*}
which leads to the desired result.
\end{proof}
\begin{rem}
    When $p\leq1$, the sufficiency still holds, but the necessary fails, as shown by the counterexample in \cite{cao2024}.
\end{rem}
\section{representation theorem}
In this section, we establish a representation theorem that writes the operator $T$ as a sum of model operators in the expectation sense. We firstly recall the definitions of compact shift and paraproduct in \cite{cao2024}.
\begin{defn}\label{defn3.1}
    Given $i,j,k\in\No$, the definition of \textit{compact bilinear dyadic shift} is given by 
    \begin{equation*}
        S_{\mathcal{D}(\sigma)}^{i,j,k}(f_1,f_2):=\sum_{K\in\mathcal{D}(\sigma)}A_K^{i,j,k}(f_1,f_2),
    \end{equation*}
    with 
    \begin{equation*}
        \ave{A_K^{i,j,k}(f_1,f_2),f_3}:=\sum_{\substack{I_1,I_2,I_3\in\mathcal{D}(\sigma)\\I_1^{(i)}=I_2^{(j)}=I_3^{(k)}=K}}F(K)a_{I_m,K}\ave{f_1,\tilde{h}_{I_1}}\ave{f_2,\tilde{h}_{I_2}}\ave{f_3,\tilde{h}_{I_3}},
    \end{equation*} 
    where $F\in\mathcal{F}_0$, $\tilde{h}_{I_m}\in\{h_{I_m}^0,h_{I_m}\}$, $m=1,2,3$, at least two of them are cancellative and the coefficients $a_{I_m,K}$ satisfy
    \[
    \abs{a_{I_m,K}}\lesssim \frac{\abs{I_1}^{1/2}\abs{I_2}^{1/2}\abs{I_3}^{1/2}}{\abs{K}^2}.
    \]
\end{defn}
\begin{defn}
    We say that $\pi_{b,\sigma}$ is a \textit{bilinear dyadic paraproduct} if 
    \[
    \ave{\pi_{b,\sigma}(f_1,f_2),f_3}=\sum_{I\in\mathcal{D}(\sigma)}\ave{f_1}_I\ave{f_2}_I\ave{\Delta_Ib,f_3},
    \]
    or other symmetric forms(i.e. the role of $f_3$ is replaced by $f_1$ or $f_2$).
\end{defn}
The compactness of these model operators was established in \cite{cao2024}.
\begin{lem}\cite[Theorem 1.7]{cao2024}\label{compactness of shifts}
    For all $i,j,k\in\No$, $\Exp_\sigma S_{\mathcal{D}(\sigma)}^{i,j,k}$ is compact from $L^{p_1}(w_1^{p_1})\times L^{p_2}(w_2^{p_2})$ to $L^{p}(w^{p})$ for all $p_1,p_2\in(1,\infty]$ for all $(w_1,w_2)\in A_{(p_1,p_2)}$, where $1/p=1/p_1+1/p_2>0$ and $w=w_1w_2$.
\end{lem}
\begin{lem}\cite[Theorem 1.8]{cao2024}\label{compactness of paraproducts}
    Given $b\in\mathrm{CMO}(\R^n)$, then $\Exp_\sigma\pi_{b,\sigma}$ is compact from $L^{p_1}(w_1^{p_1})\times L^{p_2}(w_2^{p_2})$ to $L^{p}(w^{p})$ for all $p_1,p_2\in(1,\infty]$ for all $(w_1,w_2)\in A_{(p_1,p_2)}$, where $1/p=1/p_1+1/p_2>0$ and $w=w_1w_2$.
\end{lem}
As pointed out in the paragraph before \cite[Remark 2.8]{AMV20222}, the kernel regularity plays no role in the paraproduct terms. Consequently, the result above concerning the paraproducts continues to hold without modification in the present setting. Thus, it remains to adapt the shift terms to the weaker regularity assumption on the kernel.
Inspired by \cite{AMV20222,LMV2021}, we introduce the \textit{modified compact bilinear dyadic shifts}.
\begin{defn}
    For $k\in\No$, we say that $Q_{k,\sigma}$ is a \textit{modified compact bilinear shift of $k$ complexity}, if it has the form 
    \begin{align*}
    &\ave{Q_{k,\sigma}(f_1,f_2),f_3}=\sum_{K\in\mathcal{D}(\sigma)}\sum_{\substack{I_1,I_2,I_3\in\mathcal{D}(\sigma)\\I_1^{(k)}=I_2^{(k)}=I_3^{(k)}=K}}F(K)a_{I_m,K}\ave{f_3,h_{I_3}}\\
    &\hspace*{7cm}\times\big(\ave{f_1,h_{I_1}^0}\ave{f_2,h_{I_2}^0}-\ave{f_1,h_{I_3}^0}\ave{f_2,h_{I_3}^0}\big)
    \end{align*} 
     for some $F\in\mathcal{F}_0$, or other symmetrical forms (i.e. the role of $f_3$ is replaced by $f_1$ or $f_2$), or
    \begin{align}\label{3.1}
    \ave{Q_{k,\sigma}(f_1,f_2),f_3}=\sum_{K\in\mathcal{D}(\sigma)}\sum_{\substack{I_1,I_2,I_3\in\mathcal{D}(\sigma)\\I_1^{(k)}=I_2^{(k)}=I_3^{(k)}=K}}F(K)a_{I_m,K}\ave{f_1,\tilde{h}_{I_1}}\ave{f_2,\tilde{h}_{I_2}}\ave{f_3,\tilde{h}_{I_3}},
    \end{align} 
    where $\tilde{h}_{I_m}\in\{h_{I_m}^0,h_{I_m}\}, m=1,2,3$, at least two of them are cancellative, and the coefficients $a_{I_m,K}$ satisfy
    \[
    \abs{a_{I_m,K}}\lesssim \frac{\abs{I_1}^{1/2}\abs{I_2}^{1/2}\abs{I_3}^{1/2}}{\abs{K}^2}.
    \]
\end{defn}
Note that the second form (\ref{3.1}) is a special case of Definition \ref{defn3.1}. In contrast to the original definition of the modified shift in \cite{AMV20222,LMV2021}, our definition contains an additional bounded factor $F(K)$. Since $F(K)$ is uniformly bounded, we have 
$$F(K)\cdot \abs{a_{I_m,K}}\lesssim \frac{\abs{I_1}^{1/2}\abs{I_2}^{1/2}\abs{I_3}^{1/2}}{\abs{K}^2},$$ 
therefore the boundedness estimates previously established for modified shifts remain valid in the present setting. In particular, we obtain
\begin{lem}\label{boundedness of shifts}\cite[Proposition 2.13]{AMV20222}
Let $Q_{k,\sigma}$ be a modified compact bilinear shift of $k$ complexity, then it holds that 
\begin{equation*}
    \norm{\Exp_\sigma Q_{k,\sigma}}{L^{4}\times L^{4}\to L^2}\lesssim \Norm{F}{L^\infty}(1+k)^{1/2},
\end{equation*} 
where the implicit constant is independent of the complexity parameters.
\end{lem}
This $k^{1/2}$ growth above is precisely what allows the boundedness $T1$ theorem in \cite{AMV20222} to be established under the $\log\frac{1}{2}$-Dini regularity condition. After making the aforementioned preparations, we are now ready to state our representation theorem.
\begin{thm}\label{representation}
Let $T$ be a bilinear operator associated with a compact bilinear Calder\'on-Zygmund kernel with $\log\frac{1}{2}$-Dini condition. Assume that $T$ satisfies the weak compactness property and the $T1$ conditions, namely,
\[
S(1,1)\in\BMO \text{ for all }S\in\{T,T^{*1},T^{*2}\},
\] 
then there holds that 
\begin{equation}
    \ave{T(f_1,f_2),f_3}=\Exp_\sigma\left[\sum_{k\geq0}\omega(2^{-k})\ave{P_{k,\sigma}(f_1,f_2),f_3}+\sum_{i=1}^3\ave{\pi_{b_i,\sigma}(f_1,f_2),f_3}\right],
\end{equation}
where $P_{k,\sigma}$ denotes a finite sum of modified compact bilinear shifts $Q_{k,\sigma}$, with the number of summands bounded independently of $k$ and $\sigma$ and $\pi_{b_i,\sigma}$ is the bilinear paraproduct with $b_i\in\{T(1,1),T^{*1}(1,1),T^{*2}(1,1)\}$, $i=1,2,3$.
\end{thm}
\begin{proof}
    Under the stated assumptions, the existing bilinear $T1$ theorem provides the boundedness needed to justify the limits in the martingale truncations. We establish the identity initially for finite linear combinations of indicators of cubes. Moreover, Lemma \ref{boundedness of shifts} and the summability condition
\[
\sum_{k\ge0}\omega(2^{-k})\sqrt{1+k}<\infty
\]
ensure that the averaged shift series converges in the operator norm from $L^4\times L^4$ to $L^2$. Applying the martingale difference decomposition to each function $f_j$ gives
\begin{align*}
     \bave{T(f_1,f_2),f_3}&=\Exp_\sigma\sum_{I_1,I_2,I_3\in\mathcal{D}(\sigma)} \bave{T(\Delta_{I_1}f_1,\Delta_{I_2}f_2),\Delta_{I_3}f_3}\\
    &=\Exp_\sigma\Big(\sum_{\ell(I_1),\ell(I_2)>\ell(I_3)}+\sum_{\ell(I_2),\ell(I_3)>\ell(I_1)}+\sum_{\ell(I_3),\ell(I_1)>\ell(I_2)}+\sum_{\ell(I_1)>\ell(I_2)=\ell(I_3)}\\
    &\hspace*{2cm}+\sum_{\ell(I_2)>\ell(I_3)=\ell(I_1)}+\sum_{\ell(I_3)>\ell(I_1)=\ell(I_2)}+\sum_{\ell(I_1)=\ell(I_2)=\ell(I_3)}\Big)\\
    &=:\Sigma_1+\Sigma_2+\cdots +\Sigma_7.
\end{align*}
It suffices to consider $\Sigma_1$, since $\Sigma_2,\Sigma_3$ are symmetrical and $\Sigma_4,\Sigma_5,\Sigma_6,\Sigma_7$ are simpler due to their more cancellation. Write
\begin{align*}
    \Sigma_1&=\Exp_\sigma\sum_{\ell(I_1),\ell(I_2)>\ell(I_3)}\bave{T(\Delta_{I_1}f_1,\Delta_{I_2}f_2),\Delta_{I_3}f_3}\\
    &=\Exp_\sigma\sum_{\ell(I_1)=\ell(I_2)=\ell(I_3)}\bave{T(E_{I_1}f_1,E_{I_2}f_2),\Delta_{I_3}f_3}\\
    &=\Exp_\sigma\sum_{\ell(I_1)=\ell(I_2)=\ell(I_3)}\bave{T(h_{I_1}^0,h_{I_2}^0),\Delta_{I_3}f_3}\Big(\ave{f_1,h_{I_1}^0}\ave{f_2,h_{I_2}^0}-\ave{f_1,h_{I_3}^0}\ave{f_2,h_{I_3}^0}\Big)\\
    &\hspace*{4.5cm}+\Exp_\sigma\sum_{\ell(I_1)=\ell(I_2)=\ell(I_3)}\bave{T(h_{I_1}^0,h_{I_2}^0),\Delta_{I_3}f_3}\ave{f_1,h_{I_3}^0}\ave{f_2,h_{I_3}^0}.
  \end{align*}
The last sum can be written as 
\[
\Exp_\sigma\sum_{I_3\in\mathcal{D}(\sigma)}\bave{T(1,1),\Delta_{I_3}f_3}\ave{f_1}_{I_3}\ave{f_2}_{I_3},
\]
which is already a bilinear paraproduct. To handle the first term, set $I_3=I,I_1=I\dotplus n_1,I_2=I\dotplus n_2$. Then it can be written as 
 \begin{equation*}
    \sum_{(n_1,n_2)\in\Z^n\times\Z^n\backslash\{(0,0)\}}\sum_{\substack{I\in\mathcal{D}(\sigma)}}=\sum_{k=2}^\infty\sum_{\substack{(n_1,n_2)\in\Z^n\times\Z^n\backslash\{(0,0)\}\\\max\{\abs{n_1},\abs{n_2}\}\in(2^{k-3},2^{k-2}]}}\sum_{\substack{I\in\mathcal{D}(\sigma)}}.
  \end{equation*}
  Invoking the goodness yields 
  \begin{align*}
    &\Exp_\sigma\sum_{k=2}^\infty\sum_{\substack{(n_1,n_2)\in\Z^n\times\Z^n\backslash\{(0,0)\}\\\max\{\abs{n_1},\abs{n_2}\}\in(2^{k-3},2^{k-2}]}}\sum_{\substack{I\in\mathcal{D}(\sigma)}}\\
    &=2^n\Exp_\sigma\sum_{k=2}^\infty\sum_{\substack{(n_1,n_2)\in\Z^n\times\Z^n\backslash\{(0,0)\}\\\max\{\abs{n_1},\abs{n_2}\}\in(2^{k-3},2^{k-2}]}}\sum_{I\in\mathcal{D}(\sigma,k)}=2^n\Exp_\sigma\sum_{k=2}^\infty\sum_{K\in\mathcal{D}(\sigma)}\sum_{\substack{I_1,I_2,I_3\in\mathcal{D}(\sigma)\\I_1^{(k)}=I_2^{(k)}=I_3^{(k)}=K}}\gamma_k(I_1,I_2,I_3),
  \end{align*}
  where $\gamma_k(I_1,I_2,I_3)$ equals $1$ if $I_3$ is $k$-good and the corresponding translations satisfy $2^{k-3}<\max\{|n_1|,|n_2|\}\le2^{k-2}$, and equals zero otherwise.
  The remaining $\Sigma_i$ can be dealt with similarly. From
  $$\abs{\gamma_k(I_1,I_2,I_3)\ave{T(h_{I_1}^i,h_{I_2}^j),h_{I_3}^l}}\leq\abs{\ave{T(h_{I_1}^i,h_{I_2}^j),h_{I_3}^l}},$$
it remains to estimate the coefficient $\abs{\ave{T(h_{I_1}^i,h_{I_2}^j),h_{I_3}^l}}$, where at least one Haar function is cancellative. We summarize this part of the proof in the following lemma. Note that it is only in this step that the regularity of the kernel comes into play.
  \end{proof}
\begin{lem}\label{coei}
    Let $T$ be a bilinear operator associated with a compact bilinear Calder\'on-Zygmund kernel satisfying $\log\frac{1}{2}$-Dini condition and satisfying the weak compactness property. Let $n_1,n_2\in\Z^n$. If $n_1=n_2=0$, we set $k=0$. Otherwise, we let $k\geq2$ be the unique integer such that 
    \[
    \max\{\abs{n_1},\abs{n_2}\}\in(2^{k-3},2^{k-2}].
    \]
    Then for any cube $I\in\mathcal{D}(\sigma,k)$ $(I\in\mathcal{D}(\sigma)\,\text{ for }k=0)$, with $K=I^{(k)}$ and with at least one of $\{h_I^l,h_{I\dotplus n_1}^i,h_{I\dotplus n_2}^j\}$ being cancellative, we have 
    \begin{equation*}
        \abs{\ave{T(h_{I\dotplus n_1}^i,h_{I\dotplus n_2}^j),h_{I}^l}}\lesssim\omega(2^{-k})F(K)\frac{\abs{I}^{1/2}\abs{I\dotplus n_1}^{1/2}\abs{I\dotplus n_2}^{1/2}}{\abs{K}^2}
    \end{equation*}
    for some $F\in\mathcal{F}_0.$ 
\end{lem}
\begin{proof}
    By duality, we may assume $l\neq0$. We split the proof into three cases according to the relative positions of $I,I\dotplus n_1, I\dotplus n_2$.

    {\bf Separated}, i.e., $m:=\max\{\abs{n_1},\abs{n_2}\}>1$. Then $m\in(2^{k-3},2^{k-2}]$ for some $k\geq3$, and we have 
    \[
    \frac{1}{8}\ell(K)=2^{k-3}\ell(I)\leq(m-1)\ell(I)\leq\abs{x-y}+\abs{x-z}\leq 2(m+1)\ell(I)\leq 2(2^{k-2}+1)\ell(I)\leq\ell(K),
    \]
    for all $x\in I,y\in I\dotplus n_1,z\in I\dotplus n_2$. Consequently, we can write 
    \begin{align*}
        &\abs{\ave{T(h_{I\dotplus n_1}^i,h_{I\dotplus n_2}^j),h_{I}}}\\
        &\hspace*{1cm}\leq\int_{[I\dotplus n_1]\times [I\dotplus n_2]\times I}\abs{K(x,y,z)-K(c_I,y,z)}\frac{\ud y\ud z\ud x}{\abs{I\dotplus n_1}^{1/2}\abs{I\dotplus n_2}^{1/2}\abs{I}^{1/2}}\\
        &\hspace*{1cm}\lesssim \frac{\omega(2^{-k})}{\abs{I\dotplus n_1}^{1/2}\abs{I\dotplus n_2}^{1/2}\abs{I}^{1/2}}\int_{[I\dotplus n_1]\times [I\dotplus n_2]\times I}\frac{F(x,y,z)}{(\abs{x-y}+\abs{x-z})^{2n}}\ud y\ud z\ud x.
    \end{align*}
    From \cite[Lemma 2.30 (i)]{cao2024}, we have
    \[
    F(x,y,z)\leq F_1(\abs{x-y}+\abs{x-z})F_2(\abs{x-y}+\abs{x-z})F_3\left(1+\frac{\abs{x+y}+\abs{x+z}}{1+\abs{x-y}+\abs{x-z}}\right),
    \]
    where $F_1$ is monotone increasing, $F_2,F_3$ are monotone decreasing and $(F_1,F_2,F_3)\in\mathcal{F}$. Set
    \begin{equation*}
        M(K):=\inf_{x,y,z\in K}1+\frac{\abs{x+y}+\abs{x+z}}{1+\abs{x-y}+\abs{x-z}},
    \end{equation*}
    then we obtain
\begin{align*}
        &\abs{\ave{T(h_{I\dotplus n_1}^i,h_{I\dotplus n_2}^j),h_{I}^l}}\\
        &\hspace*{1cm}\lesssim\omega(2^{-k})F_1(\ell(K))F_2(\ell(K)/8)F_3(M(K))\frac{\abs{I}^{1/2}\abs{I\dotplus n_1}^{1/2}\abs{I\dotplus n_2}^{1/2}}{\abs{K}^2}.
    \end{align*}
    Let $F_{\text{sep}}(K):=F_1(\ell(K))F_2(\ell(K)/8)F_3(M(K))$, we have 
    \begin{align*}
    \lim_{\ell(K)\to 0}F_{\text{sep}}(K)&\lesssim\lim_{\ell(K)\to 0}F_1(\ell(K))=0;\\
    \lim_{\ell(K)\to \infty}F_{\text{sep}}(K)&\lesssim\lim_{\ell(K)\to \infty}F_2(\ell(K)/8)=0.
    \end{align*}
    In other words, for any $\varepsilon>0$, there exists $0<a_\varepsilon<b_\varepsilon<\infty$ such that 
    \[
    F_{\text{sep}}(K)\lesssim \varepsilon\quad\text{ for }\ell(K)<a_\varepsilon \text{ or }\ell(K)>b_\varepsilon.
    \]
    For $a_\varepsilon<\ell(K)<b_\varepsilon$ and $\abs{c_K}$ is sufficiently large,
    \begin{align*}
        M(K)=\inf_{x,y,z\in K}\frac{1+\abs{x-y}+\abs{x-z}+\abs{x+y}+\abs{x+z}}{1+\abs{x-y}+\abs{x-z}} \geq C_\varepsilon\inf_{x,y,z\in K}\frac{1+4\abs{x}}{1+\ell(K)}\geq C_\varepsilon\abs{c_K},
    \end{align*}
    Then monotonicity gives
    \begin{align*}
    \lim_{\abs{c_K}\to \infty}F_{\text{sep}}(K)&\lesssim\varepsilon+\lim_{\substack{a_\varepsilon<\ell(K)<b_\varepsilon\\\abs{c_K}\to\infty}}F_3(C_\varepsilon\abs{c_K})\lesssim\varepsilon,
    \end{align*}
    therefore the desired estimate follows by letting $\varepsilon\to 0$.

    {\bf Adjacent}, i.e., $m:=\max\{\abs{n_1},\abs{n_2}\}=1$. Without loss of generality, suppose that the first component of $n_1$ is $1$. In this case we have 
    \[\abs{x-y}+\abs{x-z}\leq 4\ell(I)=\ell(K)\text{ for all }x\in I,y\in I\dotplus n_1,z\in I\dotplus n_2.
    \]
    Then
    \begin{align*}
         &\abs{\ave{T(h_{I\dotplus n_1}^i,h_{I\dotplus n_2}^j),h_{I}}}\\
        &\hspace*{1cm}\leq\int_{[I\dotplus n_1]\times [I\dotplus n_2]\times I}\abs{K(x,y,z)}\frac{\ud y\ud z\ud x}{\abs{I\dotplus n_1}^{1/2}\abs{I\dotplus n_2}^{1/2}\abs{I}^{1/2}}\\
        &\hspace*{1cm}\lesssim\frac{1}{\abs{I\dotplus n_1}^{1/2}\abs{I\dotplus n_2}^{1/2}\abs{I}^{1/2}}\int_{[I\dotplus n_1]\times [I\dotplus n_2]\times I}\frac{F(x,y,z)}{(\abs{x-y}+\abs{x-z})^{2n}}\ud y\ud z\ud x.
        \end{align*}
The same arguments as in the separated case imply that 
\[
F(x,y,z)\leq F_1(\ell(K))F_2(\abs{x-y}+\abs{x-z})F_3(M(K)).
\]
For $x=(x_1,\dots,x_n),y=(y_1,\dots,y_n)$ and $I=I^1\times\cdots\times I^n$, we write
\begin{align*}
        &\int_{[I\dotplus n_1]\times [I\dotplus n_2]\times I}\frac{F_2(\abs{x-y}+\abs{x-z})}{(\abs{x-y}+\abs{x-z})^{2n}}\ud y\ud z\ud x\\
&\lesssim\int_{[I\dotplus n_1]\times [I\dotplus n_2]\times I}\frac{F_2(y_1-x_1)}{(\abs{x-y}+\abs{x-z})^{2n}}\ud y\ud z\ud x\lesssim\ell(I)^{n-1}\int_{I^1}\int_{I^1\dotplus 1}\frac{F_2(y_1-x_1)}{y_1-x_1}\ud y_1\ud x_1.
    \end{align*}
Let $b$ be the right endpoint of $I^1$, then taking $s=b-x_1,t=y_1-b$, the integral above becomes 
\begin{align*}
\int_0^{\ell(I)}\int_0^{\ell(I)}\frac{F_2(s+t)}{s+t}\ud s\ud t&=\int_0^{\ell(I)} F_2(u)\ud u+\int_{\ell(I)}^{2\ell(I)}\frac{2\ell(I)-u}{u}F_2(u)\ud u\\
&\leq2\int_0^{\ell(I)} F_2(u)\ud u\lesssim\ell(I)\sum_{k=1}^\infty 2^{-k}F_2(2^{-k}\ell(I)).
\end{align*}
Let $\tilde{F}_2(t):=\sum_{k\geq1}2^{-k}F_2(2^{-k}t)$, it is still non-negative, bounded, decreasing, and satisfies $\lim_{t\to\infty}\tilde{F}_2(t)=0$. Therefore we obtain 
\begin{align*}
    &\abs{\ave{T(h_{I\dotplus n_1}^i,h_{I\dotplus n_2}^j),h_{I}^l}}\lesssim F_1(\ell(K))\tilde{F}_2(\ell(K)/4)F_3(M(K))\frac{\abs{I}^{1/2}\abs{I\dotplus n_1}^{1/2}\abs{I\dotplus n_2}^{1/2}}{\abs{K}^2},
\end{align*}
with $F_{\text{ad}}(K):=F_1(\ell(K))\tilde{F}_2(\ell(K)/4)F_3(M(K))\in\mathcal{F}_0$.

{\bf Identical}, i.e., $\abs{n_1}=\abs{n_2}=0.$ Since
\begin{align*}
    \ave{T(h_{I}^i,h_{I}^j),h_{I}}=\sum_{J_1,J_2,J_3\in\mathrm{ch}(I)}\ave{T(h_{I}^i1_{J_1},h_{I}^j1_{J_2}),h_{I}1_{J_3}},
\end{align*}
the situation reduces to the adjacent case unless $J_1=J_2=J_3$. The weak compactness property then implies 
\begin{align*}
    \babs{\sum_{J\in\mathrm{ch}(I)}\ave{T(h_{I}^i1_{J},h_{I}^j1_{J}),h_{I}1_{J}}}\lesssim \sum_{J\in\mathrm{ch}(I)}F(J)\abs{J}^{-1/2}\lesssim\max_{J\in\mathrm{ch}(I)}F(J)\abs{I}^{-1/2}.
\end{align*}
Therefore, given $F\in\mathcal{F}_0$, we define a new function $F_{\text{id}}:\mathcal{Q}\to [0,\infty)$ by
\begin{equation*}
   F_{\text{id}}(I):=F_1(2\ell(I))\tilde{F}_2(\ell(I)/2)F_3(M(I))+\max_{J\in\mathrm{ch}(I)}F(J).
\end{equation*}
Clearly, $F_{\text{id}}$ also belongs to $\mathcal{F}_0$. Hence the desired result follows by taking $F$ to be the maximum of $\{F_{\text{sep}},F_{\text{ad}},F_{\text{id}}\}$.
\end{proof}
\section{Proof of main theorem}
\subsection{(2)$\Rightarrow$(1) in Theorem \ref{main thm}}
It suffices to prove that $T$ is compact, since its adjoint operators can be handled in the same manner.
\begin{lem}
Under the assumption of $(2)$ in Theorem \ref{main thm}, then the operator $T$ is bounded from 
$L^{p_1}(w_1^{p_1})\times L^{p_2}(w_2^{p_2})$ to $L^p(w^p)$
for some $1/p=1/p_1+1/p_2>0$ with $p_1,p_2\in[1,\infty]$ and for all $(w_1,w_2)\in A_{(p_1,p_2)}$, where $w=w_1w_2$.
\end{lem}
\begin{proof}
Note that a compact bilinear Calder\'on-Zygmund kernel is in particular a bilinear Calder\'on-Zygmund kernel. The weak compactness property implies the weak boundedness property, and the $\mathrm{CMO}$ conditions in Theorem 1.1 yield the corresponding $\BMO$ control. Therefore, by the bilinear $T1$ theorem with $\log\frac{1}{2}$-Dini regularity for boundedness, $T$ is bounded in the required range. Since the kernel satisfies the ordinary Dini condition, the desired weighted boundedness follows from \cite[Theorem 1.2]{LuZhang2014}.
\end{proof}
With the estimated boundedness above, Theorem \ref{extrapolation of compactness}, Theorem \ref{representation} and Lemma \ref{compactness of paraproducts} reduce to verify 
\begin{equation*}
\sum_{k=0}^\infty\omega(2^{-k})\Exp_\sigma P_{k,\sigma}\text{ is compact from }L^{4}\times L^{4}\to L^2,
\end{equation*}
which follows from the $L^{4}\times L^{4}\to L^2$ compactness of $\sum_{k=0}^\infty\omega(2^{-k})\Exp_\sigma Q_{k,\sigma}$ since $P_{k,\sigma}$ is a sum of $Q_{k,\sigma}$ and the number of the sum is uniformly bounded.
\begin{lem}\label{compactness criteria}\cite[Lemma 4.51]{cao2024}
    Let $1/p=1/p_1+1/p_2$ with $p,p_1,p_2\in(1,\infty]$ and $\{\alpha_j\}_{j\geq0}$ be a sequence of positive numbers satisfying $\sum_{j\geq0}\alpha_j<\infty$. Assume that a bilinear operator $T$ satisfies the following conditions
    \begin{enumerate}
        \item $T=\sum_{j\geq0}\alpha_jT_j$, where each $T_j$ is a bilinear operator,
        \item $\sup_{j\geq0}\norm{T_j}{L^{p_1}\times L^{p_2}\to L^p}\leq C_0<\infty$,
        \item For every $j\geq0$, $T_j$ is compact from $L^{p_1}(\R^n)\times L^{p_2}(\R^n)\to L^p(\R^n)$.
    \end{enumerate}
    Then $T$ is compact from $L^{p_1}(\R^n)\times L^{p_2}(\R^n)\to L^p(\R^n)$.
\end{lem}
The $\log\frac{1}{2}$-Dini condition on $\omega$ implies $\sum_k\omega(2^{-k})(1+k)^{1/2}<\infty$. Therefore, if we can prove the uniform boundedness and compactness of $(1+k)^{-1/2}\Exp_\sigma Q_{k,\sigma}$, then Theorem \ref{main thm} follows immediately by applying Lemma \ref{compactness criteria} with $\alpha_k=\omega(2^{-k})(1+k)^{1/2}$ and $T_k=(1+k)^{-1/2}\Exp_\sigma Q_{k,\sigma}$. 

The uniform boundedness has already been established in Lemma \ref{boundedness of shifts}. Hence, it remains to show that each modified shift is compact. Since the compactness of the standard shifts was obtained in Lemma \ref{compactness of shifts}, if we can decompose each modified compact bilinear shift as a suitable finite sum of standard ones, then the desired result follows. 

Such a decomposition has already appeared in the proof of boundedness estimates (see e.g., \cite{LMV2021}) and inevitably involves the need to replace the principal cube. However, the additional factor $F$ may depend on the random variable $\sigma$ during this procedure. We overcome this difficulty by finding a uniform control function that is independent of the random grid.
Note that the modified compact bilinear shifts are either symmetric or already in the standard forms (\ref{3.1}), so it suffices to prove the following lemma.
\begin{lem}
    Given $k\in\No$, let $Q_{k,\sigma}$ be a modified  compact bilinear shift of k complexity with the form 
    \begin{align*}
    &\ave{Q_{k,\sigma}(f_1,f_2),f_3}=\sum_{K\in\mathcal{D}(\sigma)}\sum_{\substack{I_1,I_2,I_3\in\mathcal{D}(\sigma)\\I_1^{(k)}=I_2^{(k)}=I_3^{(k)}=K}}F(K)a_{I_m,K}\ave{f_3,h_{I_3}}\\
    &\hspace*{5cm}\times\Big(\ave{f_1,h_{I_1}^0}\ave{f_2,h_{I_2}^0}-\ave{f_1,h_{I_3}^0}\ave{f_2,h_{I_3}^0}\Big).
    \end{align*}
    Then we have
    \begin{equation*}
        Q_{k,\sigma}=\sum_{i=1}^k \Big(S_{\mathcal{D(\sigma)}}^{k,k-i,k}+S_{\mathcal{D(\sigma)}}^{k-i,0,k}-S_{\mathcal{D(\sigma)}}^{1,0,i}-S_{\mathcal{D(\sigma)}}^{0,0,i}\Big).
    \end{equation*}
\end{lem}
\begin{proof}
    We write 
    \begin{align*}
        &\abs{I_3}^{-1}(\ave{f_1,h_{I_1}^0}\ave{f_2,h_{I_2}^0}-\ave{f_1,h_{I_3}^0}\ave{f_2,h_{I_3}^0})=\ave{f_1}_{I_1}\ave{f_2}_{I_2}-\ave{f_1}_{I_3}\ave{f_2}_{I_3}\\
&\hspace*{1cm}=\ave{f_1}_{I_1}(\ave{f_2}_{I_2}-\ave{f_2}_{K})+(\ave{f_1}_{I_1}-\ave{f_1}_{K})\ave{f_2}_{K}+\ave{f_1}_{K}\ave{f_2}_{K}-\ave{f_1}_{I_3}\ave{f_2}_{I_3}.
    \end{align*}
    The first two terms can be further written as 
    \begin{equation*}
        \Sigma_1:=\sum_{i=1}^k\ave{f_1}_{I_1}\ave{\Delta_{I_2^{(i)}}f_2}_{I_2}\,\text{ and }\,   \Sigma_2:=\sum_{i=1}^k\ave{\Delta_{I_1^{(i)}}f_1}_{I_1}\ave{f_2}_{K}.
    \end{equation*}
    For the last term, we have 
    \begin{align*}
        &\ave{f_1}_{K}\ave{f_2}_{K}-\ave{f_1}_{I_3}\ave{f_2}_{I_3}=-\sum_{i=1}^{k}\Big(\ave{f_1}_{I_3^{(i-1)}}\ave{f_2}_{I_3^{(i-1)}}-\ave{f_1}_{I_3^{(i)}}\ave{f_2}_{I_3^{(i)}}\Big)\\
        &\hspace*{3cm}=-\sum_{i=1}^k\ave{f_1}_{I_3^{(i-1)}}\ave{\Delta_{I_3^{(i)}}f_2}_{I_3}-\sum_{i=1}^k\ave{\Delta_{I_3^{(i)}}f_1}_{I_3}\ave{f_2}_{I_3^{(i)}}=:\Sigma_3+\Sigma_4.
    \end{align*}
    The object invoking $\Sigma_1$ can be rewritten as
    \begin{align*}
        &\sum_{K\in\mathcal{D}(\sigma)}\sum_{\substack{I_1,I_2,I_3\in\mathcal{D}(\sigma)\\I_1^{(k)}=I_2^{(k)}=I_3^{(k)}=K}}F(K)a_{I_m,K}\ave{f_3,h_{I_3}}\abs{I_3}\sum_{i=1}^k\ave{f_1}_{I_1}\ave{\Delta_{I_2^{(i)}}f_2}_{I_2}\\
&\hspace*{1cm}=\sum_{i=1}^k\sum_{K\in\mathcal{D}(\sigma)}\sum_{\substack{I_1,L_2,I_3\in\mathcal{D}(\sigma)\\I_1^{(k)}=I_3^{(k)}=L_2^{(k-i)}=K}}F(K)\ave{f_3,h_{I_3}}\ave{f_1,h_{I_1}^0}\ave{f_2,h_{L_2}}\\
&\hspace*{8cm}\times\Big(\sum_{I_2^{(i)}=L_2}\abs{I_2}^{1/2}a_{I_m,K}\ave{h_{L_2}}_{I_2}\Big).
    \end{align*} 
 One can estimate that 
    \begin{equation*}
        \Babs{\sum_{I_2^{(i)}=L_2}\abs{I_2}^{1/2}a_{I_m,K}\ave{h_{L_2}}_{I_2}}\lesssim\sum_{I_2^{(i)}=L_2}\frac{\abs{I_1}^{1/2}\abs{I_2}\abs{I_3}^{1/2}}{\abs{K}^2}\abs{L_2}^{-1/2}\lesssim\frac{\abs{I_1}^{1/2}\abs{L_2}^{1/2}\abs{I_3}^{1/2}}{\abs{K}^2},
    \end{equation*}
    hence it corresponds to $\sum_{i=1}^k S_{\mathcal{D(\sigma)}}^{k,k-i,k}$. The same argument shows that the term corresponding to $\Sigma_2$ can be written as $\sum_{i=1}^k S_{\mathcal{D(\sigma)}}^{k-i,0,k}$. 

    Finally, we handle the terms with $\Sigma_3$ and $\Sigma_4$. For $\Sigma_3$, we take $L:=I_3^{(i)}$ as the new parent, then 
    \begin{align*}
        &-\sum_{K\in\mathcal{D}(\sigma)}\sum_{\substack{I_1,I_2,I_3\in\mathcal{D}(\sigma)\\I_1^{(k)}=I_2^{(k)}=I_3^{(k)}=K}}F(K)a_{I_m,K}\abs{I_1}\ave{f_3,h_{I_3}}\sum_{i=1}^k\ave{f_1}_{I_3^{(i-1)}}\ave{\Delta_{I_3^{(i)}}f_2}_{I_3}\\
&\hspace*{1cm}=-\sum_{i=1}^k\sum_{L\in\mathcal{D}(\sigma)}\sum_{\substack{L_1,I_3\in\mathcal{D}(\sigma)\\L_1^{(1)}=I_3^{(i)}=L}}F(L^{(k-i)})\ave{f_3,h_{I_3}}\ave{f_1,h_{L_1}^0}\ave{f_2,h_L}\\
&\hspace*{5cm}\times\Big(\delta(I_3, L_1)\sum_{\substack{I_1,I_2\in\mathcal{D}(\sigma)\\I_1^{(k)}=I_2^{(k)}=I_3^{(k)}}}a_{I_m,I_3^{(k)}}\abs{I_1}\abs{L_1}^{-1/2}\ave{h_L}_{I_3}\Big),
    \end{align*}
    where $\delta(I_3, L_1)=1$ if $I_3\subset L_1$ and otherwise $\delta(I_3,L_1)=0$. The sum in bracket above can be bounded by 
    \begin{equation*}
        \sum_{\substack{I_1,I_2\in\mathcal{D}(\sigma)\\I_1^{(k)}=I_2^{(k)}=I_3^{(k)}}}\frac{\abs{I_1}\abs{I_2}\abs{I_3}^{1/2}}{\abs{K}^2\abs{L_1}^{1/2}\abs{L}^{1/2}}\lesssim\frac{\abs{L_1}^{1/2}\abs{L}^{1/2}\abs{I_3}^{1/2}}{\abs{L}^{2}}.
    \end{equation*}
    Let 
    \[
    \tilde{F}_{k-i}(L):=\sup_{\substack{L\subset K\in\mathcal{Q}\\\ell(K)=2^{k-i}\ell(L)}}F(K),
    \]
    then we have $F(L^{(k-i)})\leq \tilde{F}_{k-i}(L)$ and $\tilde{F}_{k-i}(L)$ is independent of the choice of the random variable $\sigma$. Hence it remains to verify that $\tilde{F}_{k-i}(L)$ belongs to $\mathcal{F}_0$. Since $F\in\mathcal{F}_0$ and $\ell(K)=2^{k-i}\ell(L)$, then for any $\varepsilon>0$, there exist $0<a_\varepsilon<b_\varepsilon<\infty$ such that 
    \[
    F(K)<\varepsilon\quad\text{ for }\ell(L)< a_\varepsilon\text{ or }\ell(L)>b_\varepsilon.
    \]
    When $a_\varepsilon\leq\ell(L)\leq b_\varepsilon$, note that 
    \[
    \abs{c_K-c_L}\leq C 2^{k-i}\ell(L)\leq C 2^{k-i}b_\varepsilon,
    \]
    then
    \[
    \abs{c_{K}}\geq\abs{c_L}-\babs{c_K-c_L}\geq C_\varepsilon\abs{c_L} \text{ for }\abs{c_L}\gg b_\varepsilon.
    \]
    Consequently,
    \begin{align*}
        \lim_{\substack{\abs{c_L}\to\infty}}\tilde{F}_{k-i}(L)&\lesssim\varepsilon+\lim_{\substack{a_\varepsilon\leq\ell(L)\leq b_\varepsilon\\\abs{c_L}\to\infty}}\tilde{F}_{k-i}(L)=\varepsilon,
    \end{align*}
    and then $\tilde{F}_{k-i}\in\mathcal{F}_0$ follows by taking $\varepsilon\to 0$. Finally, absorb the ratio $F(L^{(k-i)})/\widetilde F_{k-i}(L)$ into the shift coefficients, defining it to be zero when the denominator vanishes,
    then the expression invoking $\Sigma_3$ equals $-\sum_{i=1}^kS_{\mathcal{D}(\sigma)}^{1,0,i}$, and the one with $\Sigma_4$ can similarly be written as $-\sum_{i=1}^kS_{\mathcal{D}(\sigma)}^{0,0,i}$.
\end{proof}
\subsection{(3)$\Rightarrow$(2) in Theorem \ref{main thm}}

The WCP follows directly from the WCP$^*$ by H\"older's inequality. We therefore only need to prove that $T(1,1)\in\mathrm{CMO}$ provided that $T$ satisfies WCP$^*$; the other cases are similar.

For each fixed $Q\in\mathcal{Q}$, write 
\begin{equation*}
    T(1,1)=T(1_{3Q},1_{3Q})+T(1_{(3Q)^c},1_{(3Q)^c})+T(1_{(3Q)^c},1_{3Q})+T(1_{3Q},1_{(3Q)^c}),
\end{equation*}
where the decomposition is understood modulo constants, with the far part defined through kernel differences. Then WCP$^*$ implies that 
\begin{equation*}
    \frac{1}{\abs{Q}}\int_Q \abs{T(1_{3Q},1_{3Q})(x)}\ud x\leq\left(\frac{1}{\abs{Q}}\int_{3Q}\abs{T(1_{3Q},1_{3Q})(x)}^2\ud x\right)^{1/2}\lesssim F^1(Q),
\end{equation*}
for some $F^1\in\mathcal{F}_0$. For the remaining terms, when $x\in Q$ we have
\begin{align*}
      &\abs{T(1_{(3Q)^c},1_{(3Q)^c})(x)+T(1_{(3Q)^c},1_{3Q})(x)+T(1_{3Q},1_{(3Q)^c})(x)\\
      &\hspace*{2cm}-T(1_{(3Q)^c},1_{(3Q)^c})(c_Q)-T(1_{(3Q)^c},1_{3Q})(c_Q)-T(1_{3Q},1_{(3Q)^c})(c_Q)}\\
      &\leq\int_{(3Q\times3Q)^c}\abs{K(x,y,z)-K(c_Q,y,z)}\ud z\ud y\\
      &\lesssim\int_{(3Q\times3Q)^c}\omega\left(\frac{\abs{x-c_Q}}{\abs{x-y}+\abs{x-z}}\right)\frac{F(x,y,z)}{(\abs{x-y}+\abs{x-z})^{2n}}\ud z\ud y.
\end{align*}
We split $(3Q\times3Q)^c$ as follows:
\begin{equation*}
  (3Q\times3Q)^c=\bigcup_{k\geq0}(2^{k+1}\cdot3Q)\times(2^{k+1}\cdot3Q)\backslash(2^{k}\cdot3Q)\times(2^{k}\cdot3Q)=:\bigcup_{k\geq0}Q_k.
\end{equation*}
On each $Q_k$, the arguments in the proof of Lemma \ref{coei} yield
\begin{align*}
    F(x,y,z)&\leq F_1(\abs{x-y}+\abs{x-z})F_2(\abs{x-y}+\abs{x-z})F_3\left(1+\frac{\abs{x+y}+\abs{x+z}}{1+\abs{x-y}+\abs{x-z}}\right)\\
    &\leq F_1(C 2^k\ell(Q))F_2(\ell(Q))F_3(M(2^{k+1}\cdot3Q)),
\end{align*}
where $M$ denotes the same quantity as the $M$ in the proof of Lemma \ref{coei}.
Hence we obtain 
\begin{align*}
    &\sum_{k\geq0}\int_{Q_k}\omega\left(\frac{\abs{x-c_Q}}{\abs{x-y}+\abs{x-z}}\right)\frac{F(x,y,z)}{(\abs{x-y}+\abs{x-z})^{2n}}\ud z\ud y\\
&\hspace*{3cm}\lesssim\sum_{k\geq0}\omega(2^{-k})F_1(C2^k\ell(Q))F_2(\ell(Q))F_3(M(2^{k+1}\cdot3Q))=:F^2(Q).
\end{align*}
Since $\sum_{k\geq0}\omega(2^{-k})<\infty$, then for any $\varepsilon>0$, there are $0<a_\varepsilon<b_{\varepsilon}<\infty$ and $N_\varepsilon>0$ such that 
\[
\sum_{k\geq N_\varepsilon}\omega(2^{-k})<\varepsilon,\quad F_1(\ell(Q))<\varepsilon\text{ for }\ell(Q)<a_\varepsilon,\quad F_2(\ell(Q))<\varepsilon\text{ for }\ell(Q)\geq b_\varepsilon,
\]
therefore we have
\[
F^2(Q)\lesssim\sum_{k\geq N_\varepsilon}\omega(2^{-k})+\sum_{k< N_\varepsilon}\omega(2^{-k})F_1(C2^k\ell(Q))\lesssim\varepsilon,\quad\text{ for }\ell(Q)<C^{-1}2^{-N_\varepsilon}a_\varepsilon,
\]
and 
\[
F^2(Q)\lesssim\sum_{k\geq 0}\omega(2^{-k})F_2(\ell(Q))\lesssim\varepsilon,\quad\text{ for }\ell(Q)\geq b_\varepsilon.
\]

When $C^{-1}2^{-N_\varepsilon}a_\varepsilon\leq \ell(Q)<b_\varepsilon$ and $\abs{c_Q}$ is sufficiently large, since $Q$ and $2^{k+1}\cdot 3Q$ have the same center, then the argument in the proof of Lemma \ref{coei} gives 
\[
M(2^{k+1}\cdot 3Q)\geq C_\varepsilon\abs{c_{2^{k+1}\cdot 3Q}}=C_\varepsilon \abs{c_Q},\quad\text{ for }k<N_{\varepsilon}.
\]
Then we get
\begin{align*}
\lim_{\abs{c_Q}\to\infty}F^2(Q)&\lesssim\varepsilon+\lim_{\substack{C^{-1}2^{-N_\varepsilon}a_\varepsilon\leq \ell(Q)<b_\varepsilon\\\abs{c_Q}\to\infty}}F^2(Q)\lesssim\varepsilon+\lim_{\abs{c_Q}\to\infty}\sum_{k<N_\varepsilon}\omega(2^{-k})F_3(C_\varepsilon\abs{c_Q})\lesssim\varepsilon.
\end{align*}
Hence $F^2\in\mathcal{F}_0$ by letting $\varepsilon\to0$. Combining these estimates gives
\begin{align*}
     \inf_a\frac{1}{\abs{Q}}\int_Q\abs{T(1,1)(x)-a}\ud x\lesssim F^1(Q)+F^2(Q), \text{ where }F^1+F^2\in\mathcal{F}_0.
\end{align*}
Then $T(1,1)\in\mathrm{CMO}$ follows from Lemma \ref{CMO}.
\subsection{(1)$\Rightarrow$(3) in Theorem \ref{main thm}}
By symmetry, it suffices to prove that $T$ satisfies the WCP$^*$ whenever $T$ is compact.
Fix an adjacent family of dyadic grids $\{\mathcal{D}^j\}_{j=1}^{3^n}$. By the three-lattice theorem, see e.g. \cite{LN2019}, for every $Q\in\mathcal{Q}$, there exists $j\in\{1,\cdots,3^n\}$ and $\tilde{Q}\in \mathcal{D}^j$ such that $Q\subset \tilde{Q}$ and $\ell({Q})\sim \ell(\tilde{Q})$. We choose the adjacent grids so that their cubes have sidelengths $2^k$, $k\in\mathbb Z$. For each $j$, define
\[
\mathcal D_N^j
:=
\left\{I\in\mathcal D_j:
I\cap[-2^N,2^N)^n\ne\varnothing,\quad
2^{-N}\le\ell(I)\le2^N
\right\},
\]
and let $P_N^j$ denote the associated Haar projection.

Let $N$ be the smallest non-negative integer such that $\tilde{Q}\in\mathcal{D}_{2N+1}^j$. When $N\geq3$, we estimate $ \norm{T(1_Q,1_Q)1_Q}{L^2}^2$ by
\begin{align*}
    &\int \overline{T(1_Q,1_Q)(x)}T(1_Q,1_Q)(x)1_{\tilde{Q}}(x)\ud x\\
    &\hspace*{2cm}=\int (P_N^j)^\bot\overline{T(1_Q,1_Q)(x)}T(1_Q,1_Q)(x)1_{\tilde{Q}}(x)\ud x\\
    &\hspace*{6cm}+\int P_N^j\overline{T(1_Q,1_Q)(x)}T(1_Q,1_Q)(x)1_{\tilde{Q}}(x)\ud x.
\end{align*}
The first term can be bounded by 
\begin{equation*}
    \sup_{j\in\{1,\cdots,3^n\}}\norm{(P_N^j)^\bot\overline{T}}{L^{4}\times L^4\to L^2}\abs{Q},
\end{equation*}
where $\sup_{j\in\{1,\cdots,3^n\}}\norm{(P_N^j)^\bot\overline{T}}{L^{4}\times L^4\to L^2}\to0$ as $N\to\infty$ by the compactness of $T$ and an argument similar to the proof of Lemma \ref{projection} and taking the supremum over $j$ causes no difficulty, since the family of grids is fixed and finite. By duality, the second term equals 
\begin{align*}
    &\int \overline{T(1_Q,1_Q)(x)}P_N^j\big(T(1_Q,1_Q)1_{\tilde{Q}}\big)(x)\ud x\\
    &=\int \overline{T(1_Q,1_Q)(x)}P_N^j\big(T(1_Q,1_Q)1_{\tilde{Q}}-\ave{T(1_Q,1_Q)}_{\tilde{Q}}1_{\tilde{Q}}\big)(x)\ud x\\
    &\hspace*{2cm}+\int \overline{T(1_Q,1_Q)(x)}\ave{T(1_Q,1_Q)}_{\tilde{Q}}P_N^j\big(1_{\tilde{Q}}\big)(x)\ud x=:\uppercase\expandafter{\romannumeral1}+\uppercase\expandafter{\romannumeral2}.
\end{align*}
For term $\uppercase\expandafter{\romannumeral1}$, note that 
\begin{align*}
    P_N^j\big(T(1_Q,1_Q)1_{\tilde{Q}}-\ave{T(1_Q,1_Q)}_{\tilde{Q}}1_{\tilde{Q}}\big)(x)=\sum_{\substack{J\in\mathcal{D}_N^j\\J\subset \tilde{Q}}}\Delta_J(T(1_Q,1_Q))(x),
\end{align*}
and $\tilde{Q}\notin \mathcal{D}_{2N-1}^j$ means that $\tilde{Q}\cap[-2^{2N-1},2^{2N-1})^n=\emptyset$ or $\ell(\tilde{Q})<2^{-2N+1}$ or $\ell(\tilde{Q})>2^{2N-1}$.

Assume that there exists $J\in\mathcal{D}_N^j$ with $J\subset \tilde{Q}$, then it is not the case that $\ell(\tilde{Q})<2^{-2N+1}$. If $\tilde{Q}\cap[-2^{2N-1},2^{2N-1})^n=\emptyset$, then 
\begin{equation*}
   J\subset \tilde{Q}\subset \Big([-2^{2N-1},2^{2N-1})^n\Big)^c,
\end{equation*}
which contradicts $J\in\mathcal{D}_N^j$. Hence it must be $\ell(\tilde{Q})>2^{2N-1}$. 
Since $J\in\mathcal{D}_N^j$, there exists a collection $\{Q_k\}_{k=1}^{3^n}\subset \mathcal{D}^j$ with $\ell(Q_k)=2^{N}$ such that 
\begin{equation*}
    J\subset \bigcup_{k=1}^{3^n}Q_k \text{ for all }J\in\mathcal{D}_N^j \text{ with }J\subset \tilde{Q}.
\end{equation*}
Therefore,
\begin{align*}
    \abs{\uppercase\expandafter{\romannumeral1}}&\leq\left(\int\abs{T(1_Q,1_Q)(x)}^2\ud x\right)^{1/2}\cdot\Big(\int\Babs{\sum_{\substack{J\in\mathcal{D}_N^j\\J\subset \tilde{Q}}}\Delta_J(T(1_Q,1_Q))(x)}^2\ud x\Big)^{1/2}\\
    &\lesssim\abs{Q}^{1/2}\Big(\sum_{k=1}^{3^n}\sum_{J\subset Q_k}\abs{\ave{T(1_Q,1_Q),h_J}}^2\Big)^{1/2}\lesssim\abs{Q}^{1/2}\sum_{k=1}^{3^n}\abs{Q_k}^{1/2}\lesssim2^{-nN/2}\abs{Q},
\end{align*}
where we use the fact that $T$ is bounded from $L^\infty\times L^\infty$ to $\BMO$, see e.g. \cite{LuZhang2014}.
Term $\uppercase\expandafter{\romannumeral2}$ is similar. We write 
\begin{align*}
    &\abs{\uppercase\expandafter{\romannumeral2}}\lesssim\left(\int\abs{T(1_Q,1_Q)(x)}^2\ud x\right)^{1/2}\Big(\int\abs{P_N^j(1_{\tilde{Q}})(x)}^2\ud x\Big)^{1/2}\\
    &\hspace*{5cm}\lesssim\abs{Q}^{1/2}\Big(\int\abs{P_N^j(1_{\tilde{Q}})(x)}^2\ud x\Big)^{1/2},
\end{align*}
where
\begin{equation}\label{wcp2}
    P_N^j(1_{\tilde{Q}})=\sum_{\substack{J\in\mathcal{D}_N^j\\\tilde{Q}\subsetneq J}}\ave{1_{\tilde{Q}},h_J}h_J.
\end{equation}
Suppose that there exists $J\in\mathcal{D}_N^j$ satisfying $\tilde{Q}\subsetneq J$ for $\tilde{Q}\notin \mathcal{D}_{2N-1}^j$, if $\ell(\tilde{Q})>2^{2N-1}$, then
\[
\ell(J)>\ell(\tilde{Q})>2^{2N-1},
\]
which contradicts $J\in\mathcal{D}_N^j$. If $\tilde{Q}\cap[-2^{2N-1},2^{2N-1})^n=\emptyset$, then 
\[
\dist(J,[-2^N,2^N))\geq 2^{2N-1}-2\cdot2^{N}>0\,\text{ for }N\geq 3,
\]
which also contradicts $J\in\mathcal{D}_N^j$. So it must be $\ell(\tilde{Q})<2^{-2N+1}$.
Let $\{J_k\}_{k=-N}^{N}\subset \mathcal{D}^j$ be the unique sequence such that $\ell(J_k)=2^k$ and $\tilde{Q}\subsetneq J_{-N}\subsetneq \cdots\subsetneq J_{N}$, then
\begin{equation*}
    \Big(\int\abs{P_N^j(1_{\tilde{Q}})(x)}^2\ud x\Big)^{1/2}\leq\Big(\sum_{k=-N}^{N}\abs{\ave{1_{\tilde{Q}},h_{J_k}}}^2\Big)^{1/2}\lesssim\Big(\sum_{k=-N}^{N}\abs{Q}^2\abs{J_k}^{-1}\Big)^{1/2}\lesssim2^{nN/2}\abs{Q}.
\end{equation*}
Hence,
\[
\abs{\uppercase\expandafter{\romannumeral2}}\lesssim2^{nN/2}\abs{Q}^{3/2}\lesssim2^{-nN/2}\abs{Q}.
\]
Combining these estimates, we have
\begin{equation}\label{wcp1}
    \norm{T(1_Q,1_Q)1_Q}{L^2}^2\lesssim(\sup_{j\in\{1,\cdots,3^n\}}\norm{(P_N^j)^\bot\overline{T}}{L^{4}\times L^4\to L^2}+2^{-nN/2})\abs{Q},\,\text{ for }N\geq 3.
\end{equation}
Let
\begin{equation*}
    F(Q):= \left\{
  \begin{aligned}
    & 1 ,&\text{ whenever }\, N<3;\\
    & \Big(\sup_{j\in\{1,\cdots,3^n\}}\norm{(P_N^j)^\bot\overline{T}}{L^{4}\times L^4\to L^2}+2^{-nN/2}\Big)^{1/2}, &\text{ whenever }\, N\geq3,
  \end{aligned}
  \right.
  \end{equation*}
  then the boundedness of $T$ and (\ref{wcp1}) imply that
  \[
  \norm{T(1_Q,1_Q)1_Q}{L^2}\lesssim F(Q)\abs{Q}^{1/2}.
  \]
  Note that as $\ell(Q)\to0$ or $\ell(Q)\to\infty$ or $\abs{c_Q}\to\infty$, we have $N\to\infty$, and then $F(Q)\to 0$.
  Thus $F\in\mathcal{F}_0$ and the WCP$^*$ holds.
\begin{rem}
   Our main result is stated for bilinear forms. However, thanks to the well-developed multilinear representation theorem (see, e.g., \cite{AMV2022}), similar arguments carry over to the multilinear case. We omit the details and leave them to the interested reader.
\end{rem}
\subsection*{Acknowledgements}
This paper is part of the author's PhD thesis. The author thanks his doctoral advisor Kangwei Li for suggesting this project and for carefully reading the manuscript. The author also thanks Professor Cao Mingming for providing guidance during the preparation of the manuscript.
\bibliographystyle{plain}
\bibliography{ref}

@misc{cao2024,
  author        = {Cao, Mingming and Liu, Honghai and Si, Zengyan and Yabuta, K{\'o}z{\'o}},
  title         = {A characterization of compactness via bilinear {$T1$} theorem},
  year          = {2024},
  eprint        = {2404.14013},
  archivePrefix = {arXiv},
  primaryClass  = {math.CA},
  url           = {https://arxiv.org/abs/2404.14013v2}
}

@article {AMV2022,
    AUTHOR = {Airta, Emil and Martikainen, Henri and Vuorinen, Emil},
     TITLE = {Product space singular integrals with mild kernel regularity},
   JOURNAL = {J. Geom. Anal.},
  FJOURNAL = {Journal of Geometric Analysis},
    VOLUME = {32},
      YEAR = {2022},
    NUMBER = {1},
     PAGES = {Paper No. 24, 49},
      ISSN = {1050-6926,1559-002X},
   MRCLASS = {42B20},
  MRNUMBER = {4349928},
MRREVIEWER = {Fabio\ Nicola},
       DOI = {10.1007/s12220-021-00757-3},
       URL = {https://doi.org/10.1007/s12220-021-00757-3},
}

@article {GT2002,
    AUTHOR = {Grafakos, Loukas and Torres, Rodolfo H.},
     TITLE = {Multilinear {C}alder\'on-{Z}ygmund theory},
   JOURNAL = {Adv. Math.},
  FJOURNAL = {Advances in Mathematics},
    VOLUME = {165},
      YEAR = {2002},
    NUMBER = {1},
     PAGES = {124--164},
      ISSN = {0001-8708,1090-2082},
   MRCLASS = {42B25 (35S05 47G10)},
  MRNUMBER = {1880324},
MRREVIEWER = {Gerald\ B.\ Folland},
       DOI = {10.1006/aima.2001.2028},
       URL = {https://doi.org/10.1006/aima.2001.2028},
}

@article {KS1999,
    AUTHOR = {Kenig, Carlos E. and Stein, Elias M.},
     TITLE = {Multilinear estimates and fractional integration},
   JOURNAL = {Math. Res. Lett.},
  FJOURNAL = {Mathematical Research Letters},
    VOLUME = {6},
      YEAR = {1999},
    NUMBER = {1},
     PAGES = {1--15},
      ISSN = {1073-2780},
   MRCLASS = {42B20 (26A33 47G10)},
  MRNUMBER = {1682725},
MRREVIEWER = {Michael\ T.\ Lacey},
       DOI = {10.4310/MRL.1999.v6.n1.a1},
       URL = {https://doi.org/10.4310/MRL.1999.v6.n1.a1},
}

@article {LMV2021,
    AUTHOR = {Li, Kangwei and Martikainen, Henri and Vuorinen, Emil},
     TITLE = {Genuinely multilinear weighted estimates for singular
              integrals in product spaces},
   JOURNAL = {Adv. Math.},
  FJOURNAL = {Advances in Mathematics},
    VOLUME = {393},
      YEAR = {2021},
     PAGES = {Paper No. 108099, 49},
      ISSN = {0001-8708,1090-2082},
   MRCLASS = {42B20},
  MRNUMBER = {4340231},
MRREVIEWER = {Wen\ Yuan},
       DOI = {10.1016/j.aim.2021.108099},
       URL = {https://doi.org/10.1016/j.aim.2021.108099},
}

@article {LMO2019,
    AUTHOR = {Li, Kangwei and Martikainen, Henri and Ou, Yumeng and
              Vuorinen, Emil},
     TITLE = {Bilinear representation theorem},
   JOURNAL = {Trans. Amer. Math. Soc.},
  FJOURNAL = {Transactions of the American Mathematical Society},
    VOLUME = {371},
      YEAR = {2019},
    NUMBER = {6},
     PAGES = {4193--4214},
      ISSN = {0002-9947,1088-6850},
   MRCLASS = {42B20},
  MRNUMBER = {3917220},
MRREVIEWER = {Sibei\ Yang},
       DOI = {10.1090/tran/7505},
       URL = {https://doi.org/10.1090/tran/7505},
}

@article {DYY1998,
    AUTHOR = {Deng, Donggao and Yan, Lixin and Yang, Qixiang},
     TITLE = {Blocking analysis and {$T(1)$} theorem},
   JOURNAL = {Sci. China Ser. A},
  FJOURNAL = {Science in China. Series A. Mathematics},
    VOLUME = {41},
      YEAR = {1998},
    NUMBER = {8},
     PAGES = {801--808},
      ISSN = {1006-9283,1862-2763},
   MRCLASS = {42C40 (42B20 47B38)},
  MRNUMBER = {1667059},
MRREVIEWER = {Vladimir\ B.\ Vasilyev},
       DOI = {10.1007/BF02871663},
       URL = {https://doi.org/10.1007/BF02871663},
}

@incollection {F1990,
    AUTHOR = {Figiel, Tadeusz},
     TITLE = {Singular integral operators: a martingale approach},
 BOOKTITLE = {Geometry of {B}anach spaces ({S}trobl, 1989)},
    SERIES = {London Math. Soc. Lecture Note Ser.},
    VOLUME = {158},
     PAGES = {95--110},
 PUBLISHER = {Cambridge Univ. Press, Cambridge},
      YEAR = {1990},
      ISBN = {0-521-40850-4},
   MRCLASS = {42B20 (46B09 46E40 47G10)},
  MRNUMBER = {1110189},
}

@article {GH2018,
    AUTHOR = {Grau de la Herr\'an, Ana and Hyt\"onen, Tuomas},
     TITLE = {Dyadic representation and boundedness of nonhomogeneous
              {C}alder\'on-{Z}ygmund operators with mild kernel regularity},
   JOURNAL = {Michigan Math. J.},
  FJOURNAL = {Michigan Mathematical Journal},
    VOLUME = {67},
      YEAR = {2018},
    NUMBER = {4},
     PAGES = {757--786},
      ISSN = {0026-2285,1945-2365},
   MRCLASS = {42B20 (60G46)},
  MRNUMBER = {3877436},
MRREVIEWER = {Joaqu\'in\ Motos},
       DOI = {10.1307/mmj/1531447374},
       URL = {https://doi.org/10.1307/mmj/1531447374},
}

@book {Stein1993,
    AUTHOR = {Stein, Elias M.},
     TITLE = {Harmonic analysis: real-variable methods, orthogonality, and
              oscillatory integrals},
    SERIES = {Princeton Mathematical Series},
    VOLUME = {43},
      NOTE = {With the assistance of Timothy S. Murphy,
              Monographs in Harmonic Analysis, III},
 PUBLISHER = {Princeton University Press, Princeton, NJ},
      YEAR = {1993},
     PAGES = {xiv+695},
      ISBN = {0-691-03216-5},
   MRCLASS = {42-02 (35Sxx 43-02 47G30)},
  MRNUMBER = {1232192},
MRREVIEWER = {Michael\ Cowling},
}

@article {LOPTTR2009,
    AUTHOR = {Lerner, Andrei K. and Ombrosi, Sheldy and P\'erez, Carlos and
              Torres, Rodolfo H. and Trujillo-Gonz\'alez, Rodrigo},
     TITLE = {New maximal functions and multiple weights for the multilinear
              {C}alder\'on-{Z}ygmund theory},
   JOURNAL = {Adv. Math.},
  FJOURNAL = {Advances in Mathematics},
    VOLUME = {220},
      YEAR = {2009},
    NUMBER = {4},
     PAGES = {1222--1264},
      ISSN = {0001-8708,1090-2082},
   MRCLASS = {42B20},
  MRNUMBER = {2483720},
MRREVIEWER = {Philip\ T.\ Gressman},
       DOI = {10.1016/j.aim.2008.10.014},
       URL = {https://doi.org/10.1016/j.aim.2008.10.014},
}

@article {BLL2025,
    AUTHOR = {Bai, Jin and Li, Jinsong and Li, Kangwei},
     TITLE = {Bilinear exotic {C}alder\'on-{Z}ygmund operators},
   JOURNAL = {Acta Math. Sin. (Engl. Ser.)},
  FJOURNAL = {Acta Mathematica Sinica (English Series)},
    VOLUME = {41},
      YEAR = {2025},
    NUMBER = {1},
     PAGES = {355--377},
      ISSN = {1439-8516,1439-7617},
   MRCLASS = {42B20 (42B25)},
  MRNUMBER = {4855521},
MRREVIEWER = {Jan-Olav\ R\"onning},
       DOI = {10.1007/s10114-025-3589-4},
       URL = {https://doi.org/10.1007/s10114-025-3589-4},
}

@article {CZ1952,
    AUTHOR = {Calder\'on, Alberto. P. and Zygmund, Antoni.},
     TITLE = {On the existence of certain singular integrals},
   JOURNAL = {Acta Math.},
  FJOURNAL = {Acta Mathematica},
    VOLUME = {88},
      YEAR = {1952},
     PAGES = {85--139},
      ISSN = {0001-5962,1871-2509},
   MRCLASS = {42.4X},
  MRNUMBER = {52553},
MRREVIEWER = {H.\ Kober},
       DOI = {10.1007/BF02392130},
       URL = {https://afcix1b13095ec5284139sc0bkkp6nxfbq6wnwfayc.eds.tju.edu.cn/10.1007/BF02392130},
}

@article {DJ1984,
    AUTHOR = {David, Guy and Journ\'e, Jean-Lin},
     TITLE = {A boundedness criterion for generalized {C}alder\'on-{Z}ygmund
              operators},
   JOURNAL = {Ann. of Math. (2)},
  FJOURNAL = {Annals of Mathematics. Second Series},
    VOLUME = {120},
      YEAR = {1984},
    NUMBER = {2},
     PAGES = {371--397},
      ISSN = {0003-486X,1939-8980},
   MRCLASS = {42B20 (47B38)},
  MRNUMBER = {763911},
MRREVIEWER = {Yves\ Meyer},
       DOI = {10.2307/2006946},
       URL = {https://afcix1b13095ec5284139sc0bkkp6nxfbq6wnwfayc.eds.tju.edu.cn/10.2307/2006946},
}

@article {H2012,
    AUTHOR = {Hyt\"onen, Tuomas P.},
     TITLE = {The sharp weighted bound for general {C}alder\'on-{Z}ygmund
              operators},
   JOURNAL = {Ann. of Math. (2)},
  FJOURNAL = {Annals of Mathematics. Second Series},
    VOLUME = {175},
      YEAR = {2012},
    NUMBER = {3},
     PAGES = {1473--1506},
      ISSN = {0003-486X,1939-8980},
   MRCLASS = {42B20 (42B25)},
  MRNUMBER = {2912709},
MRREVIEWER = {\'Arp\'ad\ B\'enyi},
       DOI = {10.4007/annals.2012.175.3.9},
       URL = {https://afcix1b13095ec5284139sc0bkkp6nxfbq6wnwfayc.eds.tju.edu.cn/10.4007/annals.2012.175.3.9},
}

@book {Wilson2008,
    AUTHOR = {Wilson, Michael},
     TITLE = {Weighted {L}ittlewood-{P}aley theory and exponential-square
              integrability},
    SERIES = {Lecture Notes in Mathematics},
    VOLUME = {1924},
 PUBLISHER = {Springer, Berlin},
      YEAR = {2008},
     PAGES = {xiv+224},
      ISBN = {978-3-540-74582-2},
   MRCLASS = {42B25 (42B15 42B20 47B38 47G10)},
  MRNUMBER = {2359017},
MRREVIEWER = {Caroline\ P.\ Sweezy},
}

@article {Uchi1978,
    AUTHOR = {Uchiyama, Akihito},
     TITLE = {On the compactness of operators of {H}ankel type},
   JOURNAL = {Tohoku Math. J. (2)},
  FJOURNAL = {The Tohoku Mathematical Journal. Second Series},
    VOLUME = {30},
      YEAR = {1978},
    NUMBER = {1},
     PAGES = {163--171},
      ISSN = {0040-8735,2186-585X},
   MRCLASS = {47B37 (44A25)},
  MRNUMBER = {467384},
MRREVIEWER = {Richard\ Rochberg},
       DOI = {10.2748/tmj/1178230105},
       URL = {https://doi.org/10.2748/tmj/1178230105},
}

@article {LN2019,
    AUTHOR = {Lerner, Andrei K. and Nazarov, Fedor},
     TITLE = {Intuitive dyadic calculus: the basics},
   JOURNAL = {Expo. Math.},
  FJOURNAL = {Expositiones Mathematicae},
    VOLUME = {37},
      YEAR = {2019},
    NUMBER = {3},
     PAGES = {225--265},
      ISSN = {0723-0869,1878-0792},
   MRCLASS = {42B20 (42B25)},
  MRNUMBER = {4007575},
MRREVIEWER = {Luz\ Roncal},
       DOI = {10.1016/j.exmath.2018.01.001},
       URL = {https://doi.org/10.1016/j.exmath.2018.01.001},
}

@article {V2015,
    AUTHOR = {Villarroya, Paco},
     TITLE = {A characterization of compactness for singular integrals},
   JOURNAL = {J. Math. Pures Appl. (9)},
  FJOURNAL = {Journal de Math\'ematiques Pures et Appliqu\'ees. Neuvi\`eme
              S\'erie},
    VOLUME = {104},
      YEAR = {2015},
    NUMBER = {3},
     PAGES = {485--532},
      ISSN = {0021-7824,1776-3371},
   MRCLASS = {42B25 (42C40 47G10)},
  MRNUMBER = {3383175},
MRREVIEWER = {Oscar\ Blasco},
       DOI = {10.1016/j.matpur.2015.03.006},
       URL = {https://doi.org/10.1016/j.matpur.2015.03.006},
}

@article {OV2017,
    AUTHOR = {Olsen, Jan-Fredrik and Villarroya, Paco},
     TITLE = {Endpoint estimates for compact {C}alder\'on-{Z}ygmund
              operators},
   JOURNAL = {Rev. Mat. Iberoam.},
  FJOURNAL = {Revista Matem\'atica Iberoamericana},
    VOLUME = {33},
      YEAR = {2017},
    NUMBER = {4},
     PAGES = {1285--1308},
      ISSN = {0213-2230,2235-0616},
   MRCLASS = {42B20 (42B25)},
  MRNUMBER = {3729600},
MRREVIEWER = {Israel\ P.\ Rivera-R\'ios},
       DOI = {10.4171/RMI/972},
       URL = {https://doi.org/10.4171/RMI/972},
}

@article {HL2023,
    AUTHOR = {Hyt\"onen, Tuomas and Lappas, Stefanos},
     TITLE = {Extrapolation of compactness on weighted spaces},
   JOURNAL = {Rev. Mat. Iberoam.},
  FJOURNAL = {Revista Matem\'atica Iberoamericana},
    VOLUME = {39},
      YEAR = {2023},
    NUMBER = {1},
     PAGES = {91--122},
      ISSN = {0213-2230,2235-0616},
   MRCLASS = {47B38 (35S05 42B20 42B35 46B70)},
  MRNUMBER = {4571600},
MRREVIEWER = {Junyan\ Zhao},
       DOI = {10.4171/rmi/1325},
       URL = {https://doi.org/10.4171/rmi/1325},
}

@article {HL2022,
    AUTHOR = {Hyt\"onen, Tuomas and Lappas, Stefanos},
     TITLE = {Extrapolation of compactness on weighted spaces: bilinear
              operators},
   JOURNAL = {Indag. Math. (N.S.)},
  FJOURNAL = {Koninklijke Nederlandse Akademie van Wetenschappen.
              Indagationes Mathematicae. New Series},
    VOLUME = {33},
      YEAR = {2022},
    NUMBER = {2},
     PAGES = {397--420},
      ISSN = {0019-3577,1872-6100},
   MRCLASS = {47B38 (42B20 42B35 47H60)},
  MRNUMBER = {4383118},
MRREVIEWER = {Ting\ Chen},
       DOI = {10.1016/j.indag.2021.09.007},
       URL = {https://doi.org/10.1016/j.indag.2021.09.007},
}

@article {COY2022,
    AUTHOR = {Cao, Mingming and Olivo, Andrea and Yabuta, K\^oz\^o},
     TITLE = {Extrapolation for multilinear compact operators and
              applications},
   JOURNAL = {Trans. Amer. Math. Soc.},
  FJOURNAL = {Transactions of the American Mathematical Society},
    VOLUME = {375},
      YEAR = {2022},
    NUMBER = {7},
     PAGES = {5011--5070},
      ISSN = {0002-9947,1088-6850},
   MRCLASS = {42B25 (42B20 42B35)},
  MRNUMBER = {4439498},
MRREVIEWER = {Jianwei-Urbain\ Yang},
       DOI = {10.1090/tran/8645},
       URL = {https://doi.org/10.1090/tran/8645},
}

@article {BT2013,
    AUTHOR = {B\'enyi, \'Arp\'ad and Torres, Rodolfo H.},
     TITLE = {Compact bilinear operators and commutators},
   JOURNAL = {Proc. Amer. Math. Soc.},
  FJOURNAL = {Proceedings of the American Mathematical Society},
    VOLUME = {141},
      YEAR = {2013},
    NUMBER = {10},
     PAGES = {3609--3621},
      ISSN = {0002-9939,1088-6826},
   MRCLASS = {42B20 (47B07)},
  MRNUMBER = {3080183},
MRREVIEWER = {Markus\ Passenbrunner},
       DOI = {10.1090/S0002-9939-2013-11689-8},
       URL = {https://doi.org/10.1090/S0002-9939-2013-11689-8},
}

@article {BDMT2015,
    AUTHOR = {B\'enyi, \'Arp\'ad and Dami\'an, Wendol\'in and Moen, Kabe and
              Torres, Rodolfo H.},
     TITLE = {Compact bilinear commutators: the weighted case},
   JOURNAL = {Michigan Math. J.},
  FJOURNAL = {Michigan Mathematical Journal},
    VOLUME = {64},
      YEAR = {2015},
    NUMBER = {1},
     PAGES = {39--51},
      ISSN = {0026-2285,1945-2365},
   MRCLASS = {47B20 (42B25 47B07)},
  MRNUMBER = {3326579},
MRREVIEWER = {Vivien\ G.\ Miller},
       DOI = {10.1307/mmj/1427203284},
       URL = {https://doi.org/10.1307/mmj/1427203284},
}

@article {HLTY2023,
    AUTHOR = {Hyt\"onen, Tuomas and Li, Kangwei and Tao, Jin and Yang,
              Dachun},
     TITLE = {The {$L^p$}-to-{$L^q$} compactness of commutators with
              {$p>q$}},
   JOURNAL = {Studia Math.},
  FJOURNAL = {Studia Mathematica},
    VOLUME = {271},
      YEAR = {2023},
    NUMBER = {1},
     PAGES = {85--105},
      ISSN = {0039-3223,1730-6337},
   MRCLASS = {42B20 (42B25 42B35 46E30 47B47)},
  MRNUMBER = {4596725},
MRREVIEWER = {Ronghui\ Liu},
       DOI = {10.4064/sm220910-10-1},
       URL = {https://doi.org/10.4064/sm220910-10-1},
}

@misc{MT2024,
      title={On the compactness of the bi-commutator}, 
      author={Henri Martikainen and Tuomas Oikari},
      year={2024},
      eprint={2405.05460},
      archivePrefix={arXiv},
      primaryClass={math.CA},
      url={https://arxiv.org/abs/2405.05460}, 
}

@article {MS2026,
    AUTHOR = {Mishko Mitkovski and Cody B. Stockdale},
     TITLE = {On the {$T1$} theorem for compactness of
              {C}alder\'on-{Z}ygmund operators},
   JOURNAL = {Proc. Amer. Math. Soc.},
  FJOURNAL = {Proceedings of the American Mathematical Society},
    VOLUME = {154},
      YEAR = {2026},
    NUMBER = {7},
     PAGES = {2935--2947},
      ISSN = {0002-9939,1088-6826},
   MRCLASS = {42B20 (42B25 42C40)},
  MRNUMBER = {5081765},
       DOI = {10.1090/proc/17609},
       URL = {https://doi.org/10.1090/proc/17609},
}

@article{FGW2026,
  author  = {Fragkos, Anastasios and Green, A. Walton and Wick, Brett D.},
  title   = {Multilinear Wavelet Compact {T}(1) Theorem},
  journal = {Journal d'Analyse Math{\'e}matique},
  year    = {2026},
  doi     = {10.1007/s11854-026-0438-z},
  url     = {https://doi.org/10.1007/s11854-026-0438-z}
}

@article {MSWW2023,
    AUTHOR = {Mitkovski, Mishko and Stockdale, Cody B. and Wagner, Nathan A.
              and Wick, Brett D.},
     TITLE = {Riesz-{K}olmogorov type compactness criteria in function
              spaces with applications},
   JOURNAL = {Complex Anal. Oper. Theory},
  FJOURNAL = {Complex Analysis and Operator Theory},
    VOLUME = {17},
      YEAR = {2023},
    NUMBER = {3},
     PAGES = {Paper No. 40, 31},
      ISSN = {1661-8254,1661-8262},
   MRCLASS = {46B50 (46E15 46E30 47B35)},
  MRNUMBER = {4569081},
MRREVIEWER = {Rovshan\ Alifagha\ Bandaliyev},
       DOI = {10.1007/s11785-023-01346-8},
       URL = {https://doi.org/10.1007/s11785-023-01346-8},
}

@article {AMV20222,
    AUTHOR = {Airta, Emil and Martikainen, Henri and Vuorinen, Emil},
     TITLE = {U{MD}-extensions of {C}alder\'on-{Z}ygmund operators with mild
              kernel regularity},
   JOURNAL = {J. Fourier Anal. Appl.},
  FJOURNAL = {The Journal of Fourier Analysis and Applications},
    VOLUME = {28},
      YEAR = {2022},
    NUMBER = {4},
     PAGES = {Paper No. 64, 40},
      ISSN = {1069-5869,1531-5851},
   MRCLASS = {42B20},
  MRNUMBER = {4455099},
MRREVIEWER = {Huixia\ Mo},
       DOI = {10.1007/s00041-022-09960-4},
       URL = {https://doi.org/10.1007/s00041-022-09960-4},
}

@article {LuZhang2014,
  author  = {Lu, Guozhen and Zhang, Pu},
  title   = {Multilinear {Calder\'on-Zygmund} operators with kernels of {Dini}'s type and applications},
  journal = {Nonlinear Analysis},
  volume  = {107},
  pages   = {92--117},
  year    = {2014},
  doi     = {10.1016/j.na.2014.05.005}
}
\end{document}